\documentclass{amsart}

\usepackage[margin = 1in]{geometry}

\usepackage{hyperref}
\hypersetup{
	colorlinks=true,
	linkcolor=blue,
	filecolor=magenta,      
	urlcolor=cyan,
}
\usepackage[mathscr]{euscript}
\usepackage[all,cmtip]{xy}
\usepackage{cleveref}

\usepackage{tikz}
\usepackage{tikz-cd}
\usetikzlibrary{cd, arrows.meta, calc}
\usepackage{url, hyperref}
\usepackage{float}

\tikzset{->-/.style={decoration={
  markings,
  mark=at position .45 with {\arrow{>}}},postaction={decorate}}}
\usepackage{amsmath}
\usepackage{stmaryrd}
\usepackage{amssymb}
\usepackage{enumitem}
\usepackage{graphicx}
\usepackage{mathdots}
\usepackage{color}
\usepackage{diagbox}
\usepackage{array, makecell}
\usepackage{rotating}
\usepackage{amsthm}
\usepackage{dsfont}
\usepackage{adjustbox}
\usetikzlibrary{arrows}
\usepackage{mathtools}

\theoremstyle{definition}
\newtheorem{definition}{Definition}[section]

\newtheorem{example}[definition]{Example}
\newtheorem{remark}[definition]{Remark}

\newtheorem{construction}[definition]{Construction}

\theoremstyle{theorem}
\newtheorem{theorem}[definition]{Theorem}

\newtheorem{proposition}[definition]{Proposition}
\newtheorem{corollary}[definition]{Corollary}
\newtheorem{lemma}[definition]{Lemma}

\newtheorem{question}[definition]{Question}

\def\A{\mathbb{A}}

\def\Z{\mathbb{Z}}

\def\P{\mathbb{P}}

\def\Gm{\mathbb{G}_m}
\def\GL{\mathrm{GL}}

\def\cO{\mathcal{O}}

\def\cC{\mathcal{C}}

\def\cX{\mathcal{X}}
\def\cY{\mathcal{Y}}
\def\cZ{\mathcal{Z}}

\def\cV{\mathcal{V}}
\def\cU{\mathcal{U}}

\def\cP{\mathcal{P}}
\def\cI{\mathcal{I}}

\def\Aut{\mathrm{Aut}}
\def\cN{\mathcal{N}}

\def\cXDM{\mathcal{X}_{\mathrm{DM}}}
\def\cXMOD{\mathcal{X}_{\mathrm{mod}}}

\def\cXUDM{\cX_{\mathrm{UDM}}}

\def\PhiDM{{\Phi_{\mathrm{DM}}}}
\def\PhiUDM{{\Phi_{\mathrm{UDM}}}}
\def\cYMOD{\mathcal{Y}_{\mathrm{mod}}}
\def\cUMOD{\mathcal{U}_{\mathrm{mod}}}
\def\cZMOD{\mathcal{Z}_{\mathrm{mod}}}

\def\UDM{\mathrm{UDM}}
\def\Spec{\mathrm{Spec}\ \!}

\def\DM{\mathrm{DM}}
\def\cG{\mathcal{G}}
\def\cW{\mathcal{W}}
\def\cWMOD{\mathcal{W}_{\mathrm{mod}}}
\def\cJ{\mathcal{J}}
\def\cA{\mathcal{A}}
\def\cB{\mathcal{B}}

\def\colim{\mathrm{colim}\ \!}
\def\lim{\mathrm{lim}\ \!}

\title{Factorizations of Moduli Morphisms and Universal Maps to Deligne-Mumford Stacks}

\author{Alberto Landi}
\begin{document}

\begin{abstract}
    Let $\cX$ be an algebraic stack admitting a moduli space $\cXMOD$. We study the factorizations of the moduli space morphism $\cX\rightarrow\cXMOD$ to construct intermediate stacks that simplify the stacky structure of $\cX$ while retaining more structural information than $\cXMOD$.
    Under mild assumptions, we prove the existence of a universal morphism from $\cX$ to stacks satisfying well-behaved `modular properties' (such as being Deligne-Mumford, having finite inertia, or being uniformizable), and show that this universal map is itself an adequate moduli space morphism.
    We achieve this by proving that ascending chains of adequate moduli space morphisms from a Noetherian stack stabilize if they are cohomologically affine or with target Deligne-Mumford stacks.
    Finally, we demonstrate that stabilization completely fails for general adequate moduli space morphisms. We construct a simple Noetherian, Deligne-Mumford stack admitting an infinite, non-stabilizing chain of adequate moduli space morphisms, whose limit is a non-algebraic fpqc stack.
\end{abstract}

\maketitle

\tableofcontents

\section{Introduction}

A fundamental problem in algebraic geometry is the classification of isomorphism classes of objects of a given type, such as curves or vector bundles. Traditionally, this is achieved by constructing a moduli space—a scheme or algebraic space whose closed points are in bijection with these isomorphism classes. Historically, such spaces were constructed using Geometric Invariant Theory (GIT), as developed by Mumford~\cite{GIT}.

The modern approach is to construct an algebraic stack $\cX$ that parametrizes families of these objects, thereby retaining more structural information than a coarse moduli space. Nevertheless, recovering a moduli space from $\cX$ remains geometrically significant. When $\cX$ has finite inertia, this was first achieved by Keel and Mori via coarse moduli spaces (\cite{KM97,Con05,GeneralKeelMori_Ryd07}). For the general case, Alper and others introduced the notions of good and adequate moduli spaces \cite{goodmoduli_Alp13,AdequateModAlp14,ExistenceModSpacesAHLH23,LocalStructureStacks_AHR25}. This framework provides a well-behaved morphism $\pi_\cX:\cX\rightarrow\cXMOD$ that is universal among morphisms to algebraic spaces, a theory that has since proven highly successful.

However, given an algebraic stack $\cX$ admitting a moduli space $\cXMOD$, it is natural to ask whether there exist intermediate stacks that retain more of the structural information of $\cX$ than $\cXMOD$, yet remain geometrically simpler than $\cX$ itself. This amounts to studying all the possible factorizations of the moduli morphism $\pi_{\cX}$, which can be organized in a 2-category $\cC_{\cX}$. This parametrizes \emph{adequate moduli space morphisms} $\cX\rightarrow\cY$, that are relative versions of moduli spaces; see Definition~\ref{def: moduli space morphisms}. We can then ask whether any of these simplifications of the stacky structure of $\cX$ has some universal properties.

The most immediate example is a morphism $\PhiDM:\cX\rightarrow\cXDM$ that is universal among morphisms to Deligne-Mumford stacks. In this paper, we establish the existence of such intermediate stacks under relatively weak assumptions.

\begin{theorem}[Deligne-Mumford Reduction]\label{thm: intro existence XDM}
    Let $\cX$ be a quasi-separated locally-Noetherian algebraic stack admitting an adequate moduli space $\pi_{\cX}:\cX\rightarrow\cXMOD$. Then, there exists an adequate moduli space morphism $\PhiDM:\cX\rightarrow\cXDM$ that is universal among morphisms to Deligne-Mumford stacks with affine diagonal.
\end{theorem}

See Corollary~\ref{cor: DM-fication} for a slightly more general statement and its proof. Notice that $\cXDM$ has automatically finite inertia by~\cite[Theorem 8.3.2]{AdequateModAlp14}.
More generally, we address the following question.

\begin{question}\label{question: intro universal objects}
    Given a property $\cP$ of algebraic stacks, does there exist a morphism $\Phi_{\cP}:\cX\rightarrow\cX_{\cP}$ that is universal among morphisms from $\cX$ to stacks satisfying $\cP$?
\end{question}

We answer Question~\ref{question: intro universal objects} in Theorems~\ref{thm: technical existence of initial morphisms} and~\ref{thm: existence of universal morphism strongly modular case}, for a class of well-behaved properties, that we call \emph{modular properties} (Definition~\ref{def: modular property}). We also need to assume that $\pi_{\cX}$ is a good moduli space, or that $\cP$ implies the property of being a Deligne-Mumford stack. In~\S\ref{sec: counterexample adequate} we will see how things can go very wrong without the above assumptions, even for stacks that are arguably simple.

Beyond the case of Deligne-Mumford stacks, we specialize this discussion to morphisms to stacks with finite inertia (Corollary~\ref{cor: finite inertia initial map}), as well as to uniformizable stacks (\S\ref{subsec: XUDM})—algebraic stacks admitting a finite étale cover by an algebraic space. The latter case is tightly connected to the étale fundamental group of $\cX$, making the existence of a universal map to uniformizable stacks a valuable tool for studying $\pi_1(\cX)$.

Unfortunately, the formation of $\cX_{\cP}$ does not commute with base change along morphisms to $\cX_{\cP}$ or $\cXMOD$, even when the morphisms are fppf or open immersions. This is in stark contrast with what happens for moduli spaces, and shows that the construction of $\cX_{\cP}$ cannot be local in nature.
For $\cX$ of finite type over an algebraically closed field, in~\S\ref{subsec: base change properties} we completely characterize when $\cX_{\DM}$ commutes with base change, and similarly for other properties $\cP$. As an application, we recover~\cite[Proposition B.2]{ER17} over an algebraically closed field, see Corollary~\ref{cor: Edidin Rydh result}.

\subsection{The Category $\cC_{\cX}^{\cP}$}

As mentioned above, to prove Theorem~\ref{thm: intro existence XDM}, we are led to study factorizations of the morphism $\pi_{\cX}:\cX\rightarrow\cXMOD$, that satisfy some property $\cP$. Then, we can set up a 2-category $\cC_{\cX}^{\cP}$ that parametrizes those factorizations, which turns out to be equivalent to a partially ordered set (poset) by Lemma~\ref{lem: poset structure}. It is a fundamental problem to understand the poset structure of $\cC_{\cX}^{\cP}$.

An answer to the following question would provide us with a candidate for a universal morphism to algebraic stacks satisfying $\cP$.

\begin{question}\label{question: intro existence maximum}
    Under what assumptions on $\cX$ and $\cP$, does $\cC_{\cX}^{\cP}$ admit a maximum?
\end{question}

The standard way to answer such a question is to first prove the existence of maximal elements via Zorn's lemma, and then show the uniqueness. The uniqueness holds for modular properties $\cP$, as in that case $\cC_{\cX}^{\cP}$ is filtered by Lemma~\ref{lem: refinement property of CX}. The existence part boils down to answering the following questions.

\begin{question}\label{question: intro stabilization}
    Does every ascending chain in $\cC_{\cX}^{\cP}$ admit an upper bound? If $\cX$ is Noetherian, does the chain stabilize?
\end{question}

If the answer to Question \ref{question: intro stabilization} is positive, the same holds for Question \ref{question: intro existence maximum} when $\cX$ is locally-Noetherian, and we say that $\cC_{\cX}^{\cP}$ has the \emph{stabilization property}. This means that the morphisms in the chain are isomorphisms, except possibly for a finite number of them.

\begin{proposition}[Stabilization Property]\label{thm: intro stabilization property}
    Let $\cX$ be a Noetherian algebraic stack with affine diagonal, and suppose given an ascending chain in $\cC_{\cX}$ where all morphisms are good moduli space morphisms, or all stacks are Deligne-Mumford (except possibly for $\cX$). Then, the chain stabilizes.
\end{proposition}

See Proposition~\ref{prop: generalization of stabilization} for a more general statement, where $\cX$ is only assumed to be locally-Noetherian, and the property of being a Deligne-Mumford stack is replaced by the weaker property of having smooth stabilizers.

The stabilization property is used to construct a maximum $\Phi_{\cP}:\cX\rightarrow\cX_{\cP}$ of $\cC_{\cX}^{\cP}$. This does not yet guarantee that $\cX_{\cP}$ is the universal object of Question~\ref{question: intro universal objects}. The main missing piece is to construct for every morphism $f:\cX\rightarrow\cZ$ to an algebraic stack a factorization $\cX\xrightarrow{\varphi}\cY\xrightarrow{g}\cZ$ of $f$, where $\varphi$ is an adequate moduli space morphism and $g$ is representable. We can do this in a universal way, which is interesting in its own right and extends~\cite[Theorem 3.1]{AOV08b}. See Theorem~\ref{thm: existence of relative moduli spaces} for a more precise and general statement.

\begin{theorem}\label{thm: intro relative moduli space}
    Let $\cX$ be an algebraic stack with an adequate moduli space $\pi_{\cX}:\cX\rightarrow\cXMOD$. Let $f:\cX\rightarrow\cZ$ be a morphism of algebraic stacks with $\cZ$ having affine diagonal. Then, there exists a universal factorization $\cX\xrightarrow{\varphi}\cY\xrightarrow{g}\cZ$ of $f$, where $\varphi$ is an adequate moduli space morphism and $g$ is representable. Moreover, the construction commutes with base change along representable flat morphisms to $\cY$ or $\cZ$. 
\end{theorem}

When $\cX$ admits a good moduli space morphism, we can completely characterize the category $\cC_{\cX}^{\DM}$ of good moduli space morphisms with targets Deligne-Mumford stacks. This extends results of Rydh~\cite[Theorem A.1.3]{RydhAppendixATW20}.

\begin{theorem}\label{thm: intro characterization CXDM}
    Let $\cX$ be an algebraic stack with affine diagonal, admitting a good moduli space $\pi_{\cX}:\cX\rightarrow\cXMOD$ of finite presentation. Then, the category $\cC_{\cX}^{\DM}$ is equivalent to the set of open and closed substacks $N\subset\cI_{\cX}$. Moreover, every object $\cY$ of $\cC_{\cX}^{\DM}$ has affine diagonal and finite inertia.
\end{theorem}

The correspondence is given by sending any object $\varphi:\cX\rightarrow\cY$ of $\cC_{\cX}^{\DM}$ to the relative inertia $\cI_{\cX/\cY}\subset\cI_{\cX}$.
The fact that this map is injective holds more generally for adequate moduli space morphisms to Deligne-Mumford stacks, see Theorem~\ref{thm: analogous thm 722 DM case}. This theorem is an analogue of~\cite[Theorem 7.22]{LocalStructureStacks_AHR25} and generalizes~\cite[Theorems 2.2.3, 2.3.6]{ToroidalOrbifoldsATW20}. The correspondence between moduli space morphisms and relative inertia completely fails when the target is not assumed to be Deligne-Mumford, as we explain in the following subsection.

\subsection{Inadequacy of Adequate Moduli Space Morphisms}\label{subsec: intro main counterexample}

The results we mentioned above show that $\cC_{\cX}$ is very well-behaved when we restrict to good moduli space morphisms or require the targets to be Deligne-Mumford stacks. In contrast, if we require neither of these properties, we open the door to wild, pathological examples.

For instance, it is known that in general adequate moduli space morphisms do not yield a surjection between geometric stabilizers (see for example~\cite[\S4.5]{ComplexityFlatGroupoid_RRZ18}). This undermines the intuition that moduli space morphisms simplify stabilizers. However, this intuition does hold true when considering good moduli space morphisms or Deligne-Mumford targets; see Lemma~\ref{lem: surjectivity stabilizers good and smooth stab case}.

Similarly, the stabilization property fails in general. The core issue is that adequate moduli space morphisms are not generally determined by their relative inertia, see~\cite[\S A.2.3]{RydhAppendixATW20}. In~\S\ref{sec: counterexample adequate}, we construct an example of an ascending chain in $\cC_{\cX}$ that does not stabilize even if $\cX$ is Noetherian and has a very simple structure, showing that the assumptions in the key Proposition~\ref{thm: intro stabilization property} are necessary. This surprising example is essentially an infinite iteration of the one in~\cite[\S4.5]{ComplexityFlatGroupoid_RRZ18}, and it is inspired by Nagata's famous example (\cite{RationalActions_Nagata},\cite[Example 6.5.1]{QuotientSpaces_Kol97}) of a Noetherian ring with non-Noetherian ring of invariants under a finite group action.

Explicitly, let $k$ be an infinite field of positive characteristic $p$, and let $y$ be transcendental over $k$. Let $\cX$ be the Deligne-Mumford stack $\cX=B(\Z/p\Z)_{k(y)[\epsilon]/(\epsilon^2)}$, which is a trivial gerbe over its non-reduced moduli space $\cXMOD=\Spec k(y)[\epsilon]/(\epsilon^2)$. Because $\Z/p\Z$ is not linearly-reductive, the moduli space morphism $\pi_{\cX}:\cX\rightarrow\cXMOD$ is adequate, but not good.

\begin{theorem}\label{thm: intro main counterexample}
    For every integer $n\geq1$, there exists an algebraic stack $\cY_n$ and adequate moduli space morphisms $\cX\rightarrow\cY_n\rightarrow\cY_{n-1}$ with $\cY_0=\cXMOD$, such that the following chain in $\cC_{\cX}$ does not stabilize:
    \[
    \begin{tikzcd}
        \cX\arrow[r] & \ldots\arrow[r] & \cY_n\arrow[r] & \cY_{n-1}\arrow[r] & \ldots\arrow[r] & \cY_1\arrow[r] & \cXMOD
    \end{tikzcd}
    \]
    Moreover, the maps $\cY_n\rightarrow\cXMOD$ are good moduli space morphisms, and the limit $\cY=\lim\cY_n$ is not an algebraic stack. 
\end{theorem}

The proof of the above theorem is constructive and carried out in~\S\ref{sec: counterexample adequate}.

We comment on the last part of Theorem~\ref{thm: intro main counterexample}. In \S\ref{sec: counterexample adequate}, we construct the limit $\cY$ as a quotient of an algebraic space $U$ by an fpqc-algebraic group $G$. Because $G$ is not finitely presented, this quotient is not an algebraic stack, and is merely an fpqc stack (a stack admitting a representable fpqc presentation by an algebraic space, see~\S\ref{subsec: fpqc stacks}) with representable diagonal. If the presentation were of finite presentation, then $\cY$ would be an algebraic stack by Artin's theorem (\cite[Theorem 6.1]{ArtinStacks_Art74}, \cite[Tag 06DC]{Sta24}). Moreover, $\cY$ fits into a cartesian diagram
\[
\begin{tikzcd}[row sep=small, column sep=large]
    \cU\arrow[d]\arrow[r] & U\arrow[d]\\
    \cX\arrow[r] & \cY
\end{tikzcd}
\]
where $\cU\rightarrow U$ is an adequate moduli space morphism. Since adequate moduli space morphisms descend under fpqc morphisms by~\cite[Proposition 5.2.9 (2)]{AdequateModAlp14}, it is natural to view $\cX\rightarrow\cY$ itself as an adequate moduli space morphism. This shows that very simple algebraic stacks can admit adequate moduli space morphisms to fpqc targets that are not even algebraic. This decisively shows that viewing these morphisms simply as `simplifications' of the stacky structure can be highly misleading.

Finally, we note that we explicitly leverage the fact that $\cX$ is non-reduced in this construction. Whether a similar pathological chain can exist when $\cX$ is assumed to be reduced remains an open and interesting question.

\subsection{Structure of the Paper}

In \S\ref{sec: preliminaries} we cover some definitions and technical results we will need in the rest of the paper. In particular, in \S\ref{subsec: good and adequate moduli spaces} we recall the definitions and properties of good and adequate moduli space morphisms, while in \S\ref{subsec: fpqc stacks} we give the basics of fpqc stacks.
In \S\ref{sec: positive results}, we introduce the category $\cC_{\cX}^{\cP}$ of adequate moduli space morphisms satisfying a property $\cP$, and study its fundamental features. In particular, we answer Questions~\ref{question: intro existence maximum} and~\ref{question: intro stabilization}, and prove Proposition~\ref{thm: intro stabilization property} and Theorem~\ref{thm: intro characterization CXDM}.
In \S\ref{subsec: examples and consequences}, we give applications of the results of the previous sections and provide some examples. We answer Question~\ref{question: intro universal objects} and prove Theorems~\ref{thm: intro existence XDM} and~\ref{thm: intro relative moduli space}.
Finally, in \S\ref{sec: counterexample adequate}, we construct the counterexample in Theorem~\ref{thm: intro main counterexample}.

\subsection{Conventions}

For algebraic stacks, we use the conventions of~\cite{Sta24}. In particular, algebraic stacks and spaces are not assumed to be quasi-separated, unlike some other texts like~\cite{champsalgebrique}. However, by Noetherian algebraic stack we mean a quasi-separated, quasi-compact, locally-Noetherian algebraic stack, as in~\cite[Tag 0510]{Sta24}.

\subsection{Acknowledgments}
    I am indebted to David Rydh for many invaluable insights and suggestions over correspondence, particularly for suggesting Example~\ref{exm: Rydh counterexample for DM and fi} and suggesting an approach to Theorems~\ref{thm: intro characterization CXDM} and \ref{thm: analogous thm 722 DM case}. I am deeply thankful to my advisor, Dan Abramovich, for his constant support, guidance, and feedback, which substantially improved the presentation of this work. I am also very grateful to Michele Pernice for his extensive and insightful discussions and feedback. Finally, I thank Roberta Pagliaro and Angelo Vistoli for their useful comments.
\section{Preliminaries}\label{sec: preliminaries}
    In this section we collect a few definitions and results that will be used in the rest of the paper.
    
\subsection{Good and Adequate Moduli Spaces}\label{subsec: good and adequate moduli spaces}
    
    The following results are likely well-known to specialists; however, we include their proofs here for completeness and the reader's convenience.

    We start by recalling what are good and adequate moduli space morphisms. We use the convention of~\cite[Definition 1.13]{LocalStructureStacks_AHR25}, where cohomologically and adequately affine morphisms are imposed to be stable under base change.
\begin{definition}\label{def: universally adequate homomorphism}
    Let $\cX$ be an algebraic stack. A morphism $\cA\rightarrow\cB$ of quasi-coherent $\cO_{\cX}$-algebras is \emph{universally adequate} if every section $s$ of $\cB$ over a smooth morphism $\Spec R\rightarrow\cY$ has a positive power that lifts to a section of $\cA$.
\end{definition}
\begin{definition}\label{def: moduli space morphisms}
    A quasi-compact and quasi-separated morphism $f:\cX\rightarrow\cY$ of algebraic stacks is \emph{cohomologically affine} (resp.\ \emph{adequately affine}) if:
    \begin{enumerate}
        \item\label{point 1 def} $f_*$ is exact on the category of quasi-coherent $\cO_{\cX}$-modules (resp.\ for every surjection $\cA\rightarrow\cB$ of quasi-coherent $\cO_{\cX}$-algebras, $f_*\cA\rightarrow f_*\cB$ is universally adequate); and,
        \item property \eqref{point 1 def} holds after arbitrary base change $\cY'\rightarrow\cY$.
    \end{enumerate}
    We say that $f$ is a \emph{good} (resp.\ \emph{adequate}) \emph{moduli space morphism} if $f$ is cohomologically affine (resp.\ adequately affine) and $f_*\cO_{\cX}\simeq\cO_{\cY}$. If further $\cY$ is an algebraic space, we simply call it a good (resp.\ adequate) moduli space, and denote it by $\pi_{\cX}:\cX\rightarrow\cXMOD$.
\end{definition}

    All properties above commute with flat base change $\cY'\rightarrow\cY$, and are fpqc-local on the target, see~\cite[Proposition 4.7]{goodmoduli_Alp13} and~\cite[Proposition 5.2.9]{AdequateModAlp14}. In particular, $\cX\rightarrow\cY$ is a good (resp.\ adequate) moduli space morphism if and only if for every flat morphism $U\rightarrow\cY$ from an algebraic space $\cX\times_{\cY}U\rightarrow U$ is a good (resp.\ adequate) moduli space. Good moduli space morphisms commute with arbitrary base change (\cite[Proposition 4.7]{goodmoduli_Alp13}), while adequate moduli space morphisms do so only up to adequate homeomorphism.

    \begin{definition}[{\cite[Definition 3.3.1]{AdequateModAlp14}}]
        A morphism $f:\cX\rightarrow\cY$ of algebraic stacks is an \emph{adequate homeomorphism} if $f$ is a representable, integral, universal homeomorphism which is a local isomorphism at all points with residue field 0. A morphism $\cA\rightarrow\cB$ of quasi-coherent $\cO_{\cX}$-algebras is an adequate homeomorphism if $\Spec\!_{\cX}(\cB)\rightarrow\Spec\!_{\cX}(\cA)$ is.
    \end{definition}
    Given $f:\cX\rightarrow\cY$ adequate moduli space morphism and a morphism $\cZ\rightarrow\cY$, there is a factorization $\cX\times_{\cY}\cZ\xrightarrow{\varphi}\cW\xrightarrow{g}\cZ$ where $\varphi$ is adequate moduli space morphism and $g$ is an adequate homeomorphism (\cite[Proposition 5.2.9]{AdequateModAlp14}).

\subsubsection{Cancellation Properties}\label{subsec: cancellation properties}

Consider morphisms $\cX\xrightarrow{\varphi}\cY\xrightarrow{\psi}\cZ$ and their composite, and suppose two of them are adequate moduli space morphism. Is the third an adequate moduli space morphism? The following lemma provides an answer.

\begin{lemma}\label{lem: cancellation properties}
    Let $\cX\xrightarrow{\varphi}\cY\xrightarrow{\psi}\cZ$ be quasi-compact and quasi-separated morphisms between algebraic stacks.
    \begin{enumerate}
        \item\label{lem: cancellation properties 1}  If both $\psi$ and $\varphi$ are adequately affine (respectively, cohomologically affine), then so is $\psi\circ\varphi$.
        \item\label{lem: cancellation properties 2}
        Suppose $\psi$ has affine diagonal, for instance if $\cY$ has affine diagonal and $\cZ$ has separated diagonal. If $\psi\circ\varphi$ is adequately affine (respectively, cohomologically affine), then so is $\varphi$.
        \item\label{lem: cancellation properties 3}  Suppose that $\varphi$ and $\psi\circ\varphi$ are adequately affine (respectively, cohomologically affine), and that $\cO_{\cY}\rightarrow\varphi_*\cO_{\cX}$ is an isomorphism. Then, $\psi$ is adequately affine (respectively, cohomologically affine), and $\psi_*\cO_{\cY}\simeq\cO_{\cZ}$ if and only if $\psi_*\varphi_*\cO_{\cX}\simeq\cO_{\cZ}$.
    \end{enumerate}
\end{lemma}
\begin{proof}
    The fact that classes of adequately and good affine morphisms are closed under composition is~\cite[Proposition 3.10]{goodmoduli_Alp13}, \cite[Proposition 4.2.1]{AdequateModAlp14}. Part~\eqref{lem: cancellation properties 2} is~\cite[Lemma 4.2.3]{AdequateModAlp14}.
    For the third part, let $\cB_1\rightarrow\cB_2$ be a surjection of quasi-coherent $\cO_{\cY}$-algebras; we need to show that $\psi_*\cB_1\rightarrow\psi_*\cB_2$ is universally adequate. By~\cite[Lemma 5.2.6]{AdequateModAlp14},
    $\cB_i\rightarrow\varphi_*\varphi^*\cB_i$ are adequate homeomorphisms for $i=1,2$. By~\cite[Lemmas 3.4.7, 4.1.7]{AdequateModAlp14}), $\psi_*\cB_i\rightarrow\psi_*\varphi_*\varphi^*\cB_i$ are adequate homeomorphism as well. Then, the statement follows from the next Lemma~\ref{lem: adequate properties are invariant under adequate homeomorphism}. When $\varphi$ and $\psi\circ\varphi$ are cohomologically affine, $\cB_i\rightarrow\varphi_*\varphi^*\cB_i$ are isomorphisms by~\cite[Proposition 4.5]{goodmoduli_Alp13}, and the same argument shows that $\psi_*\cB_1\rightarrow\psi_*\cB_2$ is surjective, hence $\psi$ is cohomologically affine by~\cite[Lemma 4.1.5]{AdequateModAlp14}.
\end{proof}

\begin{lemma}\label{lem: adequate properties are invariant under adequate homeomorphism}
    Suppose given a commutative diagram of homomorphism of quasi-coherent $\cO_{\cX}$-algebras
    \[
    \begin{tikzcd}
        \cA_1\arrow[r]\arrow[d,"\alpha"] & \cA_2\arrow[d,"\beta"]\\
        \cA_3\arrow[r] & \cA_4.
    \end{tikzcd}
    \]
    with the horizontal arrows being adequate homeomorphism. Then, $\alpha$ is universally adequate if and only if $\beta$ is.
\end{lemma}
\begin{proof}
    This follows from~\cite[Lemma 3.4.7]{AdequateModAlp14}.
\end{proof}

\begin{remark}\label{rmk: necessity of affine diagonal}
    In Lemma~\ref{lem: cancellation properties} part~\eqref{lem: cancellation properties 2}, requiring $\psi$ to have affine diagonal is not superfluous. Indeed, let $\cY$ be a quasi-compact and quasi-separated algebraic stack whose diagonal is not affine, and admitting a good moduli space $\psi:\cY\rightarrow\Spec A$. Then, there exists a non-affine smooth presentation $\varphi:\Spec B\rightarrow\cY$ from an affine scheme, and the pair $(\varphi,\psi)$ yields a counterexample. Indeed, if $\varphi$ was adequately affine it would also be affine by the generalized Serre's criterion for affineness~\cite[Theorem 4.3.1]{AdequateModAlp14}.
\end{remark}

The following two lemmas show that affineness of the diagonal is preserved along adequate moduli space morphisms.

\begin{lemma}\label{lem: inheriting affine diagonal}
    Let $\varphi:\cX\rightarrow\cY$ be an adequate moduli space morphism between algebraic stacks, and suppose that $\cX$ has (Zariski-locally) affine diagonal. Then, $\cY$ has (Zariski-locally) affine diagonal, and $\varphi$ has affine diagonal.
\end{lemma}
\begin{proof}
    Assume that $\cX$ has affine diagonal only Zariski-locally, and let $y$ be a geometric point of $\cY$. Let $x$ be the unique closed point in the fiber of $\varphi$ over $y$, see~\cite[Theorem 5.3.1 (5)]{AdequateModAlp14}, and $\cU\subset\cX$ an open neighborhood of $x$ that has affine diagonal. The image under $\varphi$ of $\cZ:=\cX\setminus\cU$ does not contain $y$ by~\cite[Theorem 5.3.1 (4)]{AdequateModAlp14}, hence $\varphi^{-1}(\cY\setminus\varphi(\cZ))$ is a saturated open neighborhood of $x$ contained in $\cU$. Therefore, we can assume that $\cX$ has affine diagonal. 
    By assumption the composite $\cX\rightarrow\cX\times\cX\rightarrow\cY\times\cY$ is adequately affine, hence also $\cY\rightarrow\cY\times\cY$ is by Lemma~\ref{lem: cancellation properties}~\eqref{lem: cancellation properties 3}. Then, $\cY\rightarrow\cY\times\cY$ is affine by the generalized Serre's criterion for affineness~\cite[Theorem 4.3.1]{AdequateModAlp14}. Clearly, $\varphi$ has affine diagonal as well.
\end{proof}

\begin{lemma}\label{lem: inheriting affine diagonal to mod space}
    Let $\varphi:\cX\rightarrow\cY$ be an adequate moduli space morphism between algebraic stacks admitting an adequate moduli space $\cXMOD\simeq\cYMOD$. If $\cX\rightarrow\cXMOD$ has affine diagonal, then $\cY\rightarrow\cYMOD$ has affine diagonal.
\end{lemma}
\begin{proof}
    By assumption, the composite $\cX\rightarrow\cX\times_{\cXMOD}\cX\rightarrow\cY\times_{\cXMOD}\cY$ is adequately affine, hence also $\cY\rightarrow\cY\times_{\cXMOD}\cY$ is by Lemma~\ref{lem: cancellation properties}~\eqref{lem: cancellation properties 3}. Then, $\cY\rightarrow\cY\times_{\cXMOD}\cY$ is affine by~\cite[Theorem 4.3.1]{AdequateModAlp14}.
\end{proof}

\subsubsection{Surjectivity on Stabilizers along Moduli Space Morphisms}\label{subsec: surjection stabilizers}

In general, one thinks of adequate moduli space morphisms $\varphi:\cX\rightarrow\cY$ as operations that simplify the stabilizers of $\cX$. In particular, one would expect it to induce surjections between stabilizers. In general, this is false, see~\cite[\S4.5]{ComplexityFlatGroupoid_RRZ18}, and the intuition is fallacious, as we will see in~\S\ref{sec: counterexample adequate}.
However, when $\varphi$ is cohomologically affine or $\cY$ has smooth stabilizers, the behaviour is much nicer.

\begin{lemma}\label{lem: surjectivity stabilizers good and smooth stab case}
    Let $\varphi:\cX\rightarrow\cY$ an adequate moduli space morphism between quasi-separated algebraic stacks, and let $x\in|\cX|$ be a point that is closed in the fiber over its image $\varphi(x)=y\in|\cY|$. Suppose at least one of the following holds:
    \begin{enumerate}
        \item\label{lem: surjectivity stabilizers good and smooth stab case 1} $\varphi$ is a good moduli space morphism, or
        \item\label{lem: surjectivity stabilizers good and smooth stab case 2} $G_y$ is smooth.
    \end{enumerate}
    Then, $G_x\rightarrow G_y$ is surjective.
\end{lemma}
\begin{proof}
    Let $\cG_y$ be the reduced residual gerbe of $\cY$ at $y$ (\cite[Theorem B.2]{technical_Rydh11}). First, suppose we are in case~\eqref{lem: surjectivity stabilizers good and smooth stab case 1}. Then, the base change $\cX\times_{\cY}\cG_y\rightarrow\cG_y$ is a good moduli space morphism (\cite[Proposition 4.7(1)]{goodmoduli_Alp13}). As $\cG_y$ is reduced, the same holds after passing to the reduced structure of $\cX\times_{\cY}\cG_y$ (\cite[Theorem 4.16(8)]{goodmoduli_Alp13}). Let $\cG_x\hookrightarrow(\cX\times_{\cY}\cG_y)_{\mathrm{red}}$ be the residual gerbe at $x$, which is the reduced closed substack corresponding to $\{x\}\subset|\cX|$, by assumption. Since $\cG_x\rightarrow\cG_y$ is surjective, we get that $\cG_x\rightarrow\cG_y$ is a good moduli space morphism (\cite[Lemma 4.14]{goodmoduli_Alp13}). The same holds after pullback along $\Spec k(\overline{y})\rightarrow\cG_y$, that is, for the morphism $[G_{\overline{y}}/G_{\overline{x}}]\rightarrow\Spec k(\overline{y})$. Then, $G_{\overline{y}}/G_{\overline{x}}\simeq\Spec k(\overline{y})$, hence $G_x\rightarrow G_y$ is surjective.
    
    Now, assume we are in the adequate case~\eqref{lem: surjectivity stabilizers good and smooth stab case 2}. By~\cite[Proposition 5.2.9.(3), Lemma 5.2.11]{AdequateModAlp14}, the same argument shows that $\cG_x\rightarrow\cG_y$ is the composite of an adequate moduli space morphism and a universal homeomorphism. The same holds for the  base change $[G_{\overline{y}}/G_{\overline{x}}]\rightarrow\Spec k(\overline{y})$, by~\cite[Proposition 5.2.9.(3)]{AdequateModAlp14}. As $G_{\overline{y}}/G_{\overline{x}}\rightarrow\Spec k(\overline{y})$ is an adequate homeomorphism, it is affine, so we can write $G_{\overline{y}}/G_{\overline{x}}\simeq\Spec B$, for some $k(\overline{y})$-algebra $B$. By definition of adequate homeomorphism, the injective morphism $\phi:k(\overline{y})\hookrightarrow B$ is adequate, and an isomorphism if the characteristic of $k(\overline{y})$ is 0. Therefore, we can assume the characteristic $p$ of $k(\overline{y})$ to be positive. By~\cite[Lemma 3.2.3]{AdequateModAlp14}, for every $b\in B$ there is some $r>0$ and $a\in k(\overline{y})$ such that $\phi(a)=b^{p^r}$. Since $k(\overline{y})$ is algebraically closed, there exists $a_1\in k(\overline{y})$ with $a_1^{p^r}=a$, thus $(\phi(a_1)-b)^{p^r}=0$. As $G_y/G_x$ is reduced by assumption and~\cite[Proposition 5.4.1]{AdequateModAlp14}, $\phi(a_1)=b$ and $\phi$ is an isomorphism.
\end{proof}

\subsection{Fpqc Stacks}\label{subsec: fpqc stacks}

In this subsection, we briefly introduce the notion of an fpqc stack. Since this material is only used in \S\ref{sec: counterexample adequate}, the reader may safely skip it until then.

Let $\cP$ be a property of morphisms of algebraic spaces that is stable under base change. Recall that a representable morphism $f:\cX\rightarrow\cY$ of (not necessarily algebraic) stacks is said to satisfy $\cP$ if for all morphisms $U\rightarrow\cY$ from an algebraic space $U$ the pullback morphism $\cX\times_{\cY}U\rightarrow U$ has property $\cP$.

\begin{definition}\label{def: fpqc stacks}
    Let $\cY$ be a stack in the fpqc topology. We say that $\cY$ is an \emph{fpqc stack} if there exists an algebraic space $U$ and a representable fpqc morphism $\pi:U\rightarrow\cY$. We call $\pi$ an fpqc presentation.
    We say that an fpqc stack $\cY$ is \emph{quasi-algebraic} if the diagonal is representable by algebraic spaces.
\end{definition}

\begin{example}\label{exm: example fpqc gerbe}
    An example of a quasi-algebraic stack is the classifying stack of an fpqc group scheme over a base scheme $S$. It is algebraic if and only if it is also finitely presented over $S$, see~\cite[Tag 06PL]{Sta24}.
\end{example}

\begin{remark}
    Recall that if $\cY$ admits a representable fppf morphism from a scheme, then its diagonal is automatically representable, and indeed $\cY$ is an algebraic stack by Artin's Theorem~\cite[Theorem 6.1]{ArtinStacks_Art74} (see also~\cite[Tag 06DC]{Sta24}). In contrast, not all fpqc stacks are quasi-algebraic. Indeed, having an fpqc presentation only guarantees the diagonal to be representable by fpqc spaces, that is, fpqc stacks that are equivalent to a sheaf. Since fpqc spaces are not algebraic spaces in general, as shown in~\cite{stacks-blog}, the diagonal is not guaranteed to be representable.
\end{remark}

Fpqc stacks are much less common and used in practice than algebraic stacks. Nevertheless, they can naturally arise as limits of algebraic stacks. An important example is the Nori fundamental gerbe associated to a geometrically integral algebraic stack over a field, as defined by Borne and Vistoli~\cite{Norifundgerbefiberedcategory_BV12,fundamentalegerbes_BV19}; in general, this is not algebraic by Example~\ref{exm: example fpqc gerbe}.

\subsubsection{Properties of Morphisms of Fpqc Stacks}\label{subsec: finiteness properties}

Recall that given a morphism $\phi:\cX\rightarrow\cY$ of algebraic stacks, by definition, $\phi$ is of finite type (respectively, finite presentation) if and only if for all fppf presentations $Y\rightarrow\cY$ and $X\rightarrow\cX\times_{\cY}Y$, the morphism $X\rightarrow Y$ is of finite type (respectively, of finite presentation).
This is well-defined and independent of the choice of fppf presentations because being of finite type (respectively, of finite presentation) is fppf-local on the target and the source.

Although these properties are also fpqc-local on the target, locality on the source generally fails. Consequently, for fpqc stacks, one should talk about finiteness properties only for morphisms that are representable by algebraic stacks. This motivates the following natural definition.

\begin{definition}\label{def: properties of fpqc stacks}
    Let $\cX$ and $\cY$ be fpqc stacks, let $\phi:\cX\rightarrow\cY$ be a morphism representable by algebraic stacks. Let $\cP$ be a property of morphisms of algebraic stacks that is stable under base change along fpqc morphisms, and is fpqc-local on the base. We say that $\phi$ has property $\cP$ if for any fpqc presentation $\pi:U\rightarrow\cY$ by a scheme, the base change $\cX\times_{\cY}U\rightarrow U$ has property $\cP$.
\end{definition}

This definition is independent of the choice of $U$, which can also be taken to be an algebraic stack. Moreover, the extended property $\cP$ is again stable under base change along fpqc morphisms, and is fpqc local on the base. It is stable under composition and arbitrary base change if the same holds for $\cP$ when regarded as a property of morphisms between algebraic stacks.

Notice that flatness is fpqc-local on the source; thus, to define flat morphisms of fpqc stacks, there is no need to assume representability by algebraic stacks. However, we will not need to use this broader notion.

By the following lemma, Definition~\ref{def: properties of fpqc stacks} applies to any morphism from an algebraic stacks to a quasi-algebraic stack.

\begin{lemma}\label{lem: quasi-algebraicity and representability by algebraic stacks}
    Let $\cY$ be a quasi-algebraic stack and let $f:\cX\rightarrow\cY$ be a morphism from an algebraic stack. Then, $f$ is representable by algebraic stacks.
\end{lemma}
\begin{proof}
    Let $g:\cZ\rightarrow\cY$ be another morphism from an algebraic stack. The cartesian square
    \[
    \begin{tikzcd}
        \cX\times_{\cY}\cZ\arrow[r]\arrow[d] & \cX\times\cZ\arrow[d,"(f{,}g)"]\\
        \cY\arrow[r,"\Delta"] & \cY\times\cY
    \end{tikzcd}
    \]
    shows that $\cX\times_{\cY}\cZ$ is an algebraic stack.
\end{proof}

\begin{example}\label{exm: moduli for fpqc stacks}
    Recall that being a (good or adequate) moduli space morphism is a property fpqc-local on the target and stable under fpqc base change (see~\cite[Proposition 5.2.9]{AdequateModAlp14}). In particular, Definition~\ref{def: properties of fpqc stacks} extends the notion of good and adequate moduli space morphisms to morphisms of fpqc stacks that are representable by algebraic stacks.
\end{example}

\section{The Category $\cC_{\cX}^{\cP}$ and the Stabilization Property}\label{sec: positive results}

In this section we introduce and study the category $\cC_{\cX}^{\cP}$ of adequate moduli space morphisms with source $\cX$ and satisfying a property $\cP$, and we answer Questions~\ref{question: intro existence maximum},~\ref{question: intro stabilization}.
In~\S\ref{subsec: basics on category CX} we prove the first properties of $\cC_{\cX}^{\cP}$. In~\S\ref{subsec: good case} we specialize the discussion to the case of good moduli space morphism, and prove Theorem~\ref{thm: intro characterization CXDM}. The case of adequate moduli space morphisms with Deligne-Mumford target is considered in~\S\ref{subsec: DM case}. Finally, in~\S\ref{subsec: generalizations} we prove a more general form of Proposition~\ref{thm: intro stabilization property}.

\subsection{The Category of Moduli Morphisms}\label{subsec: basics on category CX}
    The main objects of study in this paper are good and adequate moduli space morphisms, which naturally form a 2-category, as in~\cite[Appendix A]{RydhAppendixATW20}.

    \begin{definition}\label{def: category mod space morphisms}
        Let $\cX$ be an algebraic stack. We denote by $\cC_{\cX}$ the 2-category such that:
        \begin{itemize}
            \item the objects are adequate moduli space morphisms $\varphi:\cX\rightarrow\cY$,
            \item the 1-morphisms between two moduli space morphisms $\varphi_1:\cX\rightarrow\cY_1$ and $\varphi_2:\cX\rightarrow\cY_2$ consist of a morphism $\phi:\cY_1\rightarrow\cY_2$ and a 2-morphism $\alpha:\phi\circ\varphi_1\xrightarrow{\sim}\varphi_2$,
            \item a 2-morphism between two 1-morphisms $(\phi,\alpha)$ and $(\phi',\alpha')$ as above is a 2-morphism $\gamma:\phi\rightarrow\phi'$ such that $\alpha'\circ\gamma=\alpha$.
        \end{itemize}

        Let $\cP$ be a property of algebraic stacks or morphisms between them. We denote by $\cC_{\cX}^{\cP}$ the full subcategory of $\cC_{\cX}$ of objects $\varphi:\cX\rightarrow\cY$ for which $\cY$ and $\varphi$ satisfy $\cP$.
    \end{definition}

    \begin{example}\label{exm: morphism properties}
        When $\cX$ has finite inertia and $\cP$ is the property $\DM$ of being a Deligne-Mumford stack, the category $\cC_{\cX}^{\DM}$ has been studied by David Rydh in~\cite[Appendix A]{RydhAppendixATW20}.
        The main properties of morphisms that we will consider are $\cP^{\mathrm{ca}}$ of being cohomologically affine, and $\cP^{\mathrm{fp}}$ of being of finite presentation. When $\cX$ has finite inertia, $\cC_{\cX}^{\cP^{\mathrm{ca}}}$ has also been studied in loc.\! cit.
    \end{example}

    It is well-known that the categorical structure of $\cC_{\cX}$ is simpler than a 2-category.

    \begin{lemma}\label{lem: poset structure}
        Let $\cX$ be an algebraic stack and $\cP$ a property of algebraic stacks. Then, $\cC_{\cX}^{\cP}$ is equivalent to a partially-ordered set (poset).
    \end{lemma}
    \begin{proof}
        Clearly, we can assume $\cP$ to be the trivial property. By~\cite[Lemma 7.23]{LocalStructureStacks_AHR25}, every pair of morphisms in $\cC_{\cX}$ with same source and target are 2-isomorphic via a unique 2-isomorphism. This immediately implies that $\cC_{\cX}$ is equivalent to the 1-category whose objects are isomorphism classes of objects, and arrows are isomorphism classes of 1-morphisms. Moreover, the same lemma implies that two morphisms in opposite directions between two objects in $\cC_{\cX}$ are isomorphisms, and the statement follows.
\end{proof}

    We will sometimes identify $\cC_{\cX}^{\cP}$ with its associated poset. When $\cX$ admits an adequate moduli space, we can rewrite $\cC_{\cX}^{\cP}$ as follows.

    \begin{lemma}
        Let $\pi_{\cX}:\cX\rightarrow\cXMOD$ be an adequate moduli space. Then, the 2-category $\cC_{\cX}$ is equivalent to the 2-category whose objects are factorizations $\cX\xrightarrow{\varphi}\cY\xrightarrow{\psi}\cXMOD$ of $\pi_{\cX}$, 1-morphisms are arrows as in $\cC_{\cX}$ that are compatible with $\psi$, and 2-arrows are the same as in $\cC_{\cX}$. Moreover, $\psi$ is always an adequate moduli space.
    \end{lemma}
    \begin{proof}
        This follows immediately from the fact that for every object $\varphi:\cX\rightarrow\cY$ of $\cX$ there exists a unique morphism $\psi:\cY\rightarrow\cXMOD$ such that $\psi\circ\varphi=\pi_{\cX}$, by~\cite[Theorem 3.12, Lemma 7.23]{LocalStructureStacks_AHR25}. In particular, all 1-morphisms of $\cC_{\cX}$ are 1-morphisms in the new category.
        The fact that $\psi$ is an adequate moduli space follows from Lemma~\ref{lem: cancellation properties}\eqref{lem: cancellation properties 3}.
    \end{proof}

    Now, we introduce a class of properties $\cP$ of stacks for which $\cC_{\cX}^{\cP}$ is well-behaved.

    \begin{definition}\label{def: modular property}
        Let $\cX$ be an algebraic stack admitting an adequate moduli space $\pi_{\cX}:\cX\rightarrow\cXMOD$.
        A property of stacks $\cP$ is said to be \emph{modular} (resp.\ \emph{strongly modular}) with respect to $\cX$ if it satisfies:
        \begin{itemize}
            \item\label{condition: moduli} $\cXMOD$ has $\cP$;
            \item\label{condition: affine morphism} if $\cZ$ has $\cP$ and $\cY\rightarrow\cZ$ is affine (resp.\ representable), then $\cY$ has $\cP$;
            \item\label{condition: refinement} if $\cY_1$, $\cY_2$ have property $\cP$ and adequate moduli space $\cXMOD$, then $\cY_1\times_{\cXMOD}\cY_2$ has $\cP$.
        \end{itemize}
    \end{definition}

    \begin{example}\label{exm: modular properties}
    Consider the following properties of algebraic stacks:
    \begin{enumerate}
        \item\label{property: quotients} the property $\cP_{\mathrm{gq}}$ (resp.\ $\UDM$) of being a global quotient (resp.\ uniformizable);
        \item\label{property: stabilizers} the property $\DM$ of being a Deligne-Mumford stack, and $\cP_{\mathrm{mult}}$ (resp.\ $\cP_{\mathrm{ab}}$) of having stabilizers of multiplicative type (resp.\ abelian stabilizers);
        \item\label{property: finite inertia} the property $\cI_{\mathrm{fin}}$ of having finite inertia over its base;
        \item\label{property: moduli map} the property $\cP_{\mathrm{mod}}$ (resp.\ $\cP_{\mathrm{good}}$) of having an adequate (resp.\ good) moduli space;
        \item\label{property: moduli map with affine diagonal} the property $\cP_{\mathrm{adm}}$ of having an adequate moduli space whose structure morphism has affine diagonal;
        \item\label{property: diagonals} the property of having a diagonal that is quasi-separated ($\Delta_{\mathrm{qs}}$), separated ($\Delta_{\mathrm{sep}}$), affine ($\Delta_{\mathrm{aff}}$), or Zariski-locally affine ($\Delta_{\mathrm{aff}}^{\mathrm{loc}}$).
    \end{enumerate}
    
    The properties in~\eqref{property: quotients} and~\eqref{property: stabilizers} are strongly modular, the properties in~\eqref{property: finite inertia},~\eqref{property: moduli map} and~\eqref{property: moduli map with affine diagonal} are modular, and the properties in~\eqref{property: diagonals} are modular provided $\cXMOD$ satisfies $\cP$, which is always true for $\Delta_{\mathrm{qs}}$ and $\Delta_{\mathrm{sep}}$. Notice that $\Delta_{\mathrm{sep}}$ is equivalent to the identity section of the inertia being a closed immersion, which in turn is equivalent to having separated inertia. In particular, $\Delta_{\mathrm{sep}}$ is a property local on $\cX$, and is implied by $\Delta_{\mathrm{aff}}^{\mathrm{loc}}$. Also, recall that $\cP_{\mathrm{mod}}\cup\Delta_{\mathrm{aff}}^{\mathrm{loc}}$ implies $\cP_{\mathrm{adm}}$, by Lemma~\ref{lem: inheriting affine diagonal}, and $\cP_{\mathrm{adm}}$ implies $\Delta_{\mathrm{sep}}$.
    \end{example}

    \begin{example}\label{exm: non-modular properties}
        Let $\cP_{\mathrm{sm}}$ be the property of having smooth stabilizers. Then, $\cP_{\mathrm{sm}}$ is not modular with respect to $\cX$. The reason is that subgroups of smooth algebraic groups can be non-smooth, e.g. $\mu_p\subset\Gm$ over a field of characteristic $p$.
    \end{example}

    By definition, modular properties are preserved under taking refinements, in the following sense. 

    \begin{lemma}\label{lem: refinement property of CX}
        Let $\cX$ be an algebraic stack admitting an adequate moduli space $\pi_{\cX}:\cX\rightarrow\cXMOD$, and $\cP$ a property modular with respect to $\cX$. Then:
        \begin{enumerate}
            \item $\cXMOD$ is the final object of $\cC_{\cX}^{\cP}$;
            \item suppose $\cP$ is modular and implies $\cP_{\mathrm{adm}}$.
            Then $\cC_{\cX}^{\cP}$ is filtered. Explicitly, given objects $\cX\xrightarrow{\varphi_i}\cY_i\xrightarrow{\psi_i}\cXMOD$ of $\cC_{\cX}^{\cP}$ for $i=1,2$, there exists an object $\cY_{1,2}$ of $\cC_{\cX}^{\cP}$ fitting into a diagram
            \[
            \begin{tikzcd}[column sep=large,row sep=small]
            	& & \cY_1\arrow[rd,"\psi_1"]\\
            	\cX\arrow[rru,bend left=20,"\varphi_1"]\arrow[rrd,bend right=20,"\varphi_2"']\arrow[r,"\varphi_{1,2}"] & \cY_{1,2}\arrow[ru,"\theta_1"]\arrow[rd,"\theta_2"']\arrow[rr,"\psi_{1,2}"] & & \cXMOD\\
            	& & \cY_2\arrow[ru,"\psi_2"']
            \end{tikzcd}
            \]
            in $\cC_{\cX}^{\cP}$. Moreover, we can choose $\cY_{1,2}$ so that the induced morphism $\cY_{1,2}\rightarrow\cY_1\times_{\cXMOD}\cY_2$ is affine.
        \end{enumerate}
    \end{lemma}
    \begin{proof}
        The first part follows from the universal property of moduli spaces. For the second, set $\cY':=\cY_1\times_{\cXMOD}\cY_2$ with an induced morphism $\varphi':\cX\rightarrow\cY'$. By Lemma~\ref{lem: cancellation properties}\eqref{lem: cancellation properties 2}, $\varphi'$ is adequately affine, hence $\varphi_{1,2}:\cX\rightarrow\Spec\!_{\cY'}(\varphi'_*\cO_{\cX})=:\cY_{1,2}$ is an adequate moduli space morphism. By the definition of modular properties, the object constructed is in $\cC_{\cX}^{\cP}$.
    \end{proof}

    The fact that $\cC_{\cX}^{\cP}$ is filtered and equivalent to a partially ordered set, implies that there is at most one maximal element, which would then be the maximum. It is an interesting question whether this maximum exists, and if it satisfies some universal property. This motivates Questions~\ref{question: intro universal objects} and~\ref{question: intro existence maximum}.

    For instance, when $\cP=\DM$, the maximum will turn out to be the initial morphism to Deligne-Mumford stacks, representing an interesting intermediate step that retains more information on the stacky structure of $\cX$ than passing to $\cXMOD$. This and other cases are studied in \S\ref{subsec: examples and consequences}.
    
    The standard way to prove that a maximal object exists is to show that any chain of objects is dominated by another object of $\cC_{\cX}^{\cP}$, and then apply Zorn's lemma. This leads us to Question~\ref{question: intro stabilization}, that asks under which assumptions any ascending chain $\cC_{\cX}^{\cP}$ admits an upper bound, or even stabilizes. In the latter case---that is, if all morphisms in any ascending chain are isomorphisms except for a finite number of them---we say that $\cC_{\cX}^{\cP}$ satisfies the stabilization property. This property is easier to study but harder to achieve.
    
    For instance, when $\cX$ is not quasi-compact, it is hard to expect the stabilization property to hold. Nevertheless, the upper bound could still be achieved by restricting to quasi-compact open subsets, prove the stabilization property, and then glue along open immersions. A priori, the object so obtained may not have the property $\cP$, so we will need to ask $\cP$ to satisfy the following condition.

    \begin{definition}\label{def: semilocal property}
        A property of stacks $\cP$ is \emph{semi-local} if for every algebraic stack $\cX$ and every directed system of quasi-compact open substacks $\{\cU_{\lambda}\}$ of $\cX$ such that $\cX=\colim \cU_{\lambda}$, we have that $\cX$ satisfies $\cP$ if and only if $\cU_{\lambda}$ does for every $\lambda$. 
    \end{definition}

    \begin{example}\label{exm: semi-local}
        The property $\DM$ of being a Deligne-Mumford stack is clearly semi-local. The property $\UDM$ of being uniformizable is not semi-local.
    \end{example}

    \begin{lemma}\label{lem: passing to quasicompact case principle}
        Let $\cX$ be an algebraic stack and $\{\cU_{\lambda}\}$ a directed system of quasi-compact open substacks of $\cX$ such that $\cX=\colim\cU_{\lambda}$. Let $\cP$ be a semi-local property, and suppose given a chain of objects $\cY_{\theta}$ in $\cC_{\cX}^{\cP}$ whose restriction to $\cU_{\lambda}$ stabilizes in $\cC_{\cU_{\lambda}}^{\cP}$ for every $\lambda$. Then, the chain $\{\cY_{\theta}\}$ admits an upper bound in $\cC_{\cX}^{\cP}$.
    \end{lemma}
    \begin{proof}
        For every $\lambda$, let $\widetilde{\cY}^{\lambda}$ be the maximum of the restricted chain $\{\cY_{\theta}|_{\cU_{\lambda}}\}$. By construction, $\widetilde{\cY}^{\mu}|_{\cU_{\lambda}}\simeq\widetilde{\cY}^{\lambda}$ whenever $\mu\geq\lambda$, where the isomorphism is unique up to unique 2-isomorphism by~\cite[Lemma 7.23]{LocalStructureStacks_AHR25}. This gives a directed system of open immersions $\widetilde{\cY}^{\lambda}\hookrightarrow\widetilde{\cY}^{\mu}$, whose colimit is an algebraic stack $\widetilde{\cY}$ receiving an adequate moduli space morphism from $\cX$. As $\cP$ is semi-local, $\widetilde{\cY}$ is in $\cC_{\cX}^{\cP}$. Finally, $\widetilde{\cY}$ dominates the chain $\{\cY_{\theta}\}$ by construction and the definition of colimit.
    \end{proof}
    \begin{remark}\label{rmk: semi-localizing a property}
        For every property $\cP$ preserved by open immersions, we can define a semi-local property $\widetilde{\cP}$ by stating that $\cX$ has $\widetilde{\cP}$ if and only if all its quasi-compact open substacks satisfy $\cP$. Clearly, when $\cX$ is quasi-compact, this simply reduces to $\cP$.
    \end{remark}

    We conclude the subsection by noting that the association $\cX\mapsto\cC_{\cX}$ is (contravariantly) functorial for stacks with adequate moduli spaces. Indeed, let $f:\cZ\rightarrow\cX$ be a morphism of algebraic stacks with adequate moduli spaces, and $\varphi:\cX\rightarrow\cY$ an object of $\cC_{\cX}^{\cP}$. Then, we get a commutative diagram with cartesian square
    \[
    \begin{tikzcd}
        \cZ\arrow[r,"\cC_{f}(\varphi)"]\arrow[dr,"f"'] & \cC_f(\cY)\arrow[r] & \cY'\arrow[d]\arrow[r] & \cZMOD\arrow[d]\\
        & \cX\arrow[r,"\varphi"] & \cY\arrow[r] & \cXMOD
    \end{tikzcd}
    \]
    where $\cC_f(\cY)=\Spec\!_{\cY'}(\varphi'_*\cO_{\cZ})$, with $\varphi':\cZ\rightarrow\cY'$. If $\cP$ implies $\cP_{\mathrm{adm}}$, then $\cC_f(\varphi)$ is an adequate moduli space morphism by Lemma~\ref{lem: cancellation properties}\eqref{lem: cancellation properties 2}.
    Moreover, any 1-morphism in $\cC_{\cX}^{\cP}$ induces one between the associated objects in $\cC_{\cZ}^{\cP}$, and similarly for 2-morphisms, proving the following lemma.

    \begin{lemma}\label{lem: functoriality CX}
        Let $f:\cZ\rightarrow\cX$ be a morphism of algebraic stacks with adequate moduli spaces, and $\cP$ a strongly modular property that implies $\cP_{\mathrm{adm}}$.
        The above construction defines a functor $\cC_f^{\cP}:\cC_{\cX}^{\cP}\rightarrow\cC_{\cZ}^{\cP}$. Moreover, it is compatible with composition of morphisms of algebraic stacks.
    \end{lemma}
    
\subsection{Case of Good Moduli Morphism}\label{subsec: good case}
    In this subsection we study $\cC_{\cX}^{\cP}$ when $\cP$ implies $\cP^{\mathrm{ca}}$, that is we restrict $\cC_{\cX}$ to objects $\varphi:\cX\rightarrow\cY$ that are good moduli space morphisms. When $\cX$ admits a good moduli space and $\cP=\Delta_{\mathrm{sep}}\cup\DM$, we identify $\cC_{\cX}^{\cP}$ with the set of open and closed subgroups of $\cI_{\cX}$. See Examples~\ref{exm: morphism properties}, \ref{exm: modular properties} for the definition of the various properties.
    As a simple consequence, we give a first answer to Question~\ref{question: intro stabilization}. This will be generalized in~\S\ref{subsec: generalizations}.

    The following result is a direct consequence of~\cite[Theorem 7.22]{LocalStructureStacks_AHR25}.

    \begin{lemma}\label{lem: inclusion CX for good moduli spaces}
        Let $\cX$ be an algebraic stack, and let $\cP$ be a property that implies $\Delta_{\mathrm{qs}}$, $\cP^{\mathrm{ca}}$ and $\cP^{\mathrm{fp}}$. Then, $\cC_{\cX}^{\cP}$ is equivalent to a partially ordered subset of the set of subgroups of the inertia $\cI_{\cX}$, by sending an object $\varphi:\cX\rightarrow\cY$ of $\cC_{\cX}^{\cP}$ to the relative inertia $\cI_{\cX/\cY}\subset\cI_{\cX}$. Moreover, the corresponding subgroups are closed (resp.\ open) if $\cP$ implies $\Delta_{\mathrm{sep}}$ (resp.\ if $\cP$ implies $\DM$).
    \end{lemma}
    \begin{proof}
        The assumptions allow us to apply~\cite[Theorem 7.22]{LocalStructureStacks_AHR25}, which implies that the map of posets $(\varphi:\cX\rightarrow\cY)\mapsto\cI_{\cX/\cY}$ is injective. The last part of the statement follows from the fact that the zero section of the inertia is closed if the inertia is separated over its base, and open if the stack is Deligne-Mumford.
    \end{proof}

    This is already enough to answer Question~\ref{question: intro stabilization} for good moduli space morphisms.

    \begin{proposition}\label{prop: stabilization good case}
        Let $\cX$ be an algebraic stack of finite presentation over a quasi-separated and locally-Noetherian algebraic space $S$.
        Let $\cP$ be a property that implies $\Delta_{\mathrm{sep}}$ and $\cP^{\mathrm{ca}}$. Assume further that $S$ is Noetherian or $\cP$ is semi-local.
        Then, every ascending chain in $\cC_{\cX}^{\cP}$ admits an upper bound, and if $S$ is Noetherian, then every ascending chain in $\cC_{\cX}^{\cP}$ stabilizes. In particular, $\cC_{\cX}^{\cP}$ is either empty or admits a maximal element.
    \end{proposition}
    \begin{proof}
        By~\cite[Theorem 6.3.3]{AdequateModAlp14}, for every object $\varphi:\cX\rightarrow\cY$ of $\cC_{\cX}^{\cP}$, the induced morphism $\cY\rightarrow S$ is of finite type, hence $\varphi$ is of finite presentation. By Lemma~\ref{lem: passing to quasicompact case principle}, we can assume that $S$ is Noetherian. By Lemma~\ref{lem: inclusion CX for good moduli spaces}, $\cC_{\cX}^{\cP}$ is equivalent to a partially ordered subset of the set of closed subgroups of the inertia. As $\cI_{\cX}\rightarrow\cX$ is of finite type, $\cI_{\cX}$ is also Noetherian, and the stabilization property follows. The existence of a maximal element follows by Zorn's lemma.
    \end{proof}

    \begin{corollary}\label{cor: existence maximum good moduli space case}
        Let $\cX$ be an algebraic stack of finite presentation over a quasi-separated locally-Noetherian algebraic space $S$, and admitting a good moduli space $\pi_{\cX}:\cX\rightarrow\cXMOD$. Let $\cP$ be a modular property that implies $\cP_{\mathrm{adm}}$.
        Suppose that $S$ is Noetherian or $\cP$ is semi-local. Then, $\cC_{\cX}^{\cP}$ admits a maximum.
    \end{corollary}
    \begin{proof}
        Since $\cP$ implies $\cP_{\mathrm{adm}}$ or $\Delta_{\mathrm{aff}}$, and $\pi_{\cX}$ is a good moduli space morphism, by Lemma~\ref{lem: cancellation properties} we have $\cC_{\cX}^{\cP}=\cC_{\cX}^{\cP\cup\cP^{\mathrm{ca}}}$. Therefore, we can apply Proposition~\ref{prop: stabilization good case} to get a maximal element. Finally, by Lemma~\ref{lem: refinement property of CX} the maximal element is unique and it is a maximum.
    \end{proof}

    We will consider applications of Corollary~\ref{cor: existence maximum good moduli space case} in~\S\ref{subsec: examples and consequences}, in particular proving the existence of initial morphisms to stacks with finite inertia.

    \subsubsection{Characterization of $\cC_{\cX}^{\DM}$}
    In the remaining part of this subsection we answer the following question in an important case: is the association $(\cX\rightarrow\cY)\mapsto(\cI_{\cX/\cY}\subset\cI_{\cX})$ bijective? This is known to be true when $\cX$ has finite inertia and $\cP=\DM\cup\Delta_{\mathrm{sep}}$, see~\cite[Theorem A.1.3]{RydhAppendixATW20}. We prove Theorem~\ref{thm: intro characterization CXDM}, which answers the above question in the case where $\cX$ admits a good moduli space and $\cP=\DM$. The strategy of the proof was suggested by David Rydh.
    The first step is to generalize~\cite[Proposition A.1.5]{RydhAppendixATW20} to stacks admitting good moduli spaces.

    \begin{proposition}\label{prop: generalization Rydh Prop A15}
        Let $\cX$ be an algebraic stack with affine diagonal, admitting a good moduli space $\pi_{\cX}:\cX\rightarrow\cXMOD$ locally of finite presentation.
        Let $\cN\subset\cI_{\cX}$ be an open and closed subgroup of the inertia. Then there is an affine, étale and surjective morphism $h:\cW\rightarrow\cX$ such that $\cI_{\cW}=h^*\cN$ as subgroups of $h^*\cI_{\cX}$.
    \end{proposition}
    \begin{proof}
        Let $x\in|\cX|$ with image $x_0\in|\cXMOD|$, and let $\cG_x$ be the residual gerbe of $\cX$ at $x$. Now, the inclusion $\cN\subset\cI_{\cX}$ yields a representable étale morphism $h_0:\cW_0\rightarrow\cG_x$, where $\cW_0$ is a gerbe over $\Spec k(x)$ with linearly reductive stabilizer $\cN_x$. Then,~\cite[Theorem 1.1]{LocalStructureStacks_AHR25} applies to give a cartesian diagram
        \[
        \begin{tikzcd}
            \cG_w=\cW_0\arrow[d]\arrow[r,"h_0"] & \cG_x\arrow[d]\\
            {[\Spec A/\GL_n]}=\cW\arrow[r,"h"] & \cX
        \end{tikzcd}
        \]
        and $w\in|\cW|$ such that $\cG_w\simeq\cW_0$, $w$ maps to $x$ and is closed in the fiber over $x_0$, and $h$ is étale and affine.
        By construction, at $w$ the inclusion $\cI_{\cW}\hookrightarrow\cI_{\cX}\times_{\cX}\cW=h^*\cI_{\cX}$ is $G_w\simeq\cN_w\subset G_x$. Now, $h$ is separated and étale, therefore $\cI_{\cW}\rightarrow\cI_{\cX}\times_{\cX}\cW$ is an open and closed immersion, as there is a cartesian diagram
        \[
        \begin{tikzcd}
            \cI_{\cW}\arrow[r]\arrow[d] & \cI_{\cX}\times_{\cX}\cW\arrow[d]\\
            \cW\arrow[r] & \cW\times_{\cX}\cW.
        \end{tikzcd}
        \]
        Therefore, the locus $\cZ\subset h^*\cI_{\cX}$ where $\cI_{\cW}\not=h^*\cN$ is the union of a number of connected components of $h^*\cI_{\cX}$. Let $p_1:\cZ\rightarrow\cW$ the projection to $\cW$, whose image does not contain $w$ by construction. As $\cX$ has unpunctured inertia (see~\cite[Definition 3.52, Theorem 4.1]{ExistenceModSpacesAHLH23}), it follows that $w\not\in\overline{p_1(\cZ)}$, as can be shown exactly as in~\cite[Proposition 4.3 (2)]{ExistenceModSpacesAHLH23}. Now, $w$ is closed in its fiber over $\cXMOD$, hence it is the unique closed point lying in the fiber over $w_0=\pi_{\cW}(w)$ in $\cWMOD\simeq\Spec A^{\GL_n}$. Then, by \cite[Theorem 4.16]{goodmoduli_Alp13}, $w_0$ is disjoint from $\pi_{\cW}(\overline{p_1(\cZ)})$. Therefore, there is an affine open neighborhood $U_0$ of $w_0$ disjoint from $\pi_{\cW}(\overline{p_1(\cZ)})$, and $\cU:=\pi_{\cW}^{-1}(U_0)\hookrightarrow\cW$ is an affine open immersion.
        In particular, after replacing $\cW$ with $\cU$ we can assume that $\cI_{\cW}\simeq h^*\cN$, but maintaining all the properties of $\cW\rightarrow\cX$. By varying $x\in|\cX|$ and taking the disjoint union, we have found $h:\cW\rightarrow\cX$ as in the statement.
    \end{proof}
    
    \begin{proof}[Proof of Theorem~\ref{thm: intro characterization CXDM}]
        Notice that by Lemma~\ref{lem: inheriting affine diagonal} all objects of $\cC_{\cX}^{\cP}$ have affine diagonal, thus satisfy $\Delta_{\mathrm{sep}}$ as well. Moreover, they satisfy $\cP^{\mathrm{fp}}$, by~\cite[Theorem 6.3.3]{AdequateModAlp14}. By Lemma~\ref{lem: inclusion CX for good moduli spaces}, it is enough to produce for every $\cN\subset\cI_{\cX}$ as in the statement a good moduli space morphism $\varphi:\cX\rightarrow\cY$ with $\cI_{\cX/\cY}=\cN$ satisfying property $\cP$.

        By Proposition~\ref{prop: generalization Rydh Prop A15}, there is an affine, étale and surjective morphism $h:\cW\rightarrow\cX$ such that $\cI_{\cW}=h^*\cN$ as subgroups of $h^*\cI_{\cX}$. As in the proof of~\cite[Theorem A.1.3]{RydhAppendixATW20}, the two projections $p_i:\cW\times_{\cX}\cW\rightarrow\cW$ are étale and inert, that is, $\cI_{\cW\times_{\cX}\cW}\simeq p_i^*\cI_{\cW}$. Moreover, both stacks admit good moduli spaces by~\cite[Lemma 4.14]{goodmoduli_Alp13}. Then, by Luna's fundamental lemma~\cite[Theorem 3.14]{LocalStructureStacks_AHR25}, we get an étale groupoid $(\cW\times_{\cX}\cW)_{\mathrm{mod}}\rightrightarrows\cWMOD$, whose quotient is a Deligne-Mumford stack $\cY$. By the properties of groupoids, we get a quotient morphism $\varphi:\cX\rightarrow\cY$ such that $\cX\times_{\cY}\cWMOD\simeq\cW$. In particular, $\varphi$ is a good moduli space morphism by~\cite[Proposition 4.7 (2)]{goodmoduli_Alp13}, and $\cI_{\cX/\cY}\simeq\cN$ since $h^*\cI_{\cX/\cY}\simeq\cI_{\cW}\simeq h^*\cN$. By Lemma~\ref{lem: inheriting affine diagonal}, the diagonal of $\cY$ is affine, and in particular separated.
        Therefore, $\varphi$ is the object of $\cC_{\cX}^{\cP}$ we were looking for.

        The finiteness of the inertia follows from the fact that $\cY$ admits a good moduli space by Lemma~\ref{lem: cancellation properties}, the diagonal is quasi-finite and separated, and from~\cite[Theorem 8.3.2]{AdequateModAlp14}.
    \end{proof}
    
\subsection{Case of Adequate Moduli Space Morphisms to Deligne-Mumford Stacks}\label{subsec: DM case}

In this subsection we replicate the discussion above when the moduli space $\pi_{\cX}:\cX\rightarrow\cXMOD$ is only assumed to be adequate. In this case, we need to assume that $\cP$ implies the property $\DM$. In particular, we give an answer to Question~\ref{question: intro stabilization} when $\cP$ implies the property $\DM$, later generalized in~\S\ref{subsec: generalizations}.

The first step is proving a version of~\cite[Theorem 7.22]{LocalStructureStacks_AHR25} for adequate moduli space morphisms with Deligne-Mumford target, as suggested by Rydh. The proof is similar to loc.\! cit., and the result is already known only when $\cX$ has finite inertia, see~\cite[Theorems 2.2.3, 2.3.6]{ToroidalOrbifoldsATW20} and~\cite[Appendix A]{RydhAppendixATW20}.

\begin{theorem}\label{thm: analogous thm 722 DM case}
    Let $\pi:\cX\rightarrow\cY$ be an adequate moduli space morphism of finite presentation between algebraic stacks. Let $\cZ$ be a Deligne-Mumford stack with quasi-separated diagonal, and $f:\cX\rightarrow\cZ$ a morphism. Then, $f$ factors through $\pi$ if and only if the induced map on inertia $\cI_{\cX/\cY}\rightarrow\pi^*\cI_{\cZ}$ factors through the identity section. Moreover, the factorization is unique up to unique 2-isomorphism and if $\cY$ is quasi-compact and quasi-separated, then the condition is equivalent to:
    \[
        \ker(\Aut_{\cX}(x)\rightarrow\Aut_{\cY}(\pi(x)))\subset\ker(\Aut_{\cX}(x)\rightarrow\Aut_{\cZ}(f(x)))
    \]
    for every closed point $x\in|\cX|$.
\end{theorem}
\begin{proof}
    The uniqueness is~\cite[Lemma 7.23]{LocalStructureStacks_AHR25}. By uniqueness, we can work smooth-locally on $\cY$, hence assume $\cY=\cXMOD$ is an affine scheme. As $\cZ$ is a Deligne-Mumford stack, we can take $p:U\rightarrow\cZ$ to be an étale morphism from an affine scheme. Let $x\in|\cX|$ be a point that is closed in the fiber over a closed point in $\cXMOD$, and denote by $\cG_x$ the residual gerbe of $x$ at $\cX$. By assumption, the morphism $\cG_x\rightarrow\cZ$ factors through $\Spec k(x)$, yielding a section to the restriction of $q:\cX\times_{\cZ}U\rightarrow\cX$ over $\cG_x$, eventually after a separable extension of $k(x)$. By assumption, the image $X_0$ of $\cG_x$ via $\pi_{\cX}$ is a closed subscheme of $\cXMOD$; let $\cXMOD^h$ be the henselization of $\cXMOD$ along $X_0$, and set $\cX^h:=\cX\times_{\cXMOD}\cXMOD^h$. Consider the induced morphism $q^h:U\times_{\cZ}\cX^h\rightarrow\cX^h$, which is again étale. By~\cite[Theorem 3.10]{LocalStructureStacks_AHR25}, the pair $(\cX^h,\cG_x)$ is henselian, hence by~\cite[Proposition 3.11]{LocalStructureStacks_AHR25} or~\cite[Proposition 5.4]{HR23}, the section of $q^h|_{\cG_x}$ extends to a section of $q^h$, yielding $\cX^h\rightarrow U$ by composition. As $\cXMOD^h$ is the limit of all affine étale neighborhoods of $X_0$, the morphism $\cXMOD^h\rightarrow U$ induced by the universal property of adequate moduli spaces extends to an étale neighborhood of $X_0$ in $\cXMOD$. Composing with $p:U\rightarrow\cZ$, this gives the desired map. The last part of the statement can be proved as in~\cite[Theorem 7.22]{LocalStructureStacks_AHR25}.
\end{proof}

Similarly to \S\ref{subsec: good case}, the above result provides an answer to Question~\ref{question: intro stabilization}.

\begin{corollary}\label{cor: inclusion CX for DM case}
Let $\cX$ be an algebraic stack of finite presentation over a quasi-separated locally-Noetherian algebraic space $S$.
Let $\cP$ be a property that implies $\DM$ (resp.\ $\Delta_{\mathrm{sep}}$ and $\DM$). Then, $\cC_{\cX}^{\cP}$ is equivalent to a partially ordered subset of the set of open (resp.\ closed and open) subgroups of the inertia $\cI_{\cX}$, by sending an object $\varphi:\cX\rightarrow\cY$ of $\cC_{\cX}^{\cP}$ to the relative inertia $\cI_{\cX/\cY}\subset\cI_{\cX}$.
\end{corollary}
\begin{proof}
    The proof is the same as Lemma~\ref{lem: inclusion CX for good moduli spaces}, using Theorem~\ref{thm: analogous thm 722 DM case} in place of~\cite[Theorem 7.22]{LocalStructureStacks_AHR25}.
\end{proof}

\begin{corollary}\label{cor: stabilization DM case}
    Let $\cX$ be an algebraic stack of finite presentation over a quasi-separated locally-Noetherian algebraic space $S$.
    Let $\cP$ be a property that implies $\Delta_{\mathrm{sep}}$ and
    $\DM$. Assume further that $\cP$ is semi-local or that $S$ is Noetherian. Then, every ascending chain in $\cC_{\cX}^{\cP}$ admits an upper bound, and if $S$ is Noetherian, then every ascending chain in $\cC_{\cX}^{\cP}$ stabilizes. In particular, $\cC_{\cX}^{\cP}$ is either empty or admits a maximal element.
\end{corollary}
\begin{proof}
    First, notice that under this assumption $\cP$ implies $\cP^{\mathrm{fp}}$ as well. Then, the proof is the same as Proposition~\ref{prop: stabilization good case}, using Corollary~\ref{cor: inclusion CX for DM case} in place of Lemma~\ref{lem: inclusion CX for good moduli spaces}.
\end{proof}

\begin{corollary}\label{cor: existence maximum DM case}
    Let $\cX$ be an algebraic stack of finite presentation over a quasi-separated locally Noetherian algebraic space $S$ admitting an adequate moduli space $\pi_{\cX}:\cX\rightarrow\cXMOD$. Let $\cP$ be a modular property that implies $\DM$ and $\cP_{\mathrm{adm}}$.
    Assume further that $S$ is Noetherian or $\cP$ is semi-local. Then, $\cC_{\cX}^{\cP}$ admits a maximum.
\end{corollary}
\begin{proof}
    The proof is the same as for Corollary~\ref{cor: existence maximum good moduli space case}.
\end{proof}

\begin{remark}\label{rmk: case of smooth stabilizers}
    By looking at the counterexample~\cite[Example A.2.3]{RydhAppendixATW20} (or~\cite[\S4.5]{ComplexityFlatGroupoid_RRZ18}) to Theorem~\ref{thm: analogous thm 722 DM case}, one could think that the only obstruction to results as~\cite[Theorem 7.22]{LocalStructureStacks_AHR25} and Theorem~\ref{thm: analogous thm 722 DM case} is the non-reducedness of the stabilizers of the target, even when they are of higher dimension. This is not true.
Indeed, recall that in~\cite[\S4.5]{ComplexityFlatGroupoid_RRZ18} the authors construct a wild DM-stack $\cX$ over $\mathbb{F}_p$ and a morphism $f:\cX\rightarrow B\mu_p$ that is trivial on inertia groups but does not factor through $\cXMOD$. Then, one could compose with the $\Gm$-torsor $B\mu_p\rightarrow B\Gm$ to obtain $g:\cX\rightarrow B\Gm$, whose target has then smooth stabilizer. By construction, $g$ induces a trivial morphism between inertia groups, but still does not factor through $\cXMOD$. Indeed, $f$ and $g$ correspond to the same line bundle $L$ over $\cX$, which is not a pullback from $\cXMOD$.
Nevertheless, the analogue of Corollary~\ref{cor: stabilization DM case} still holds, as we will see in the next subsection. 
\end{remark}

Again, we postpone to~\S\ref{subsec: examples and consequences} for examples and applications of the stabilization property.

\subsection{Generalization of the Stabilization Property}\label{subsec: generalizations}

In this subsection we prove a more general form of Corollaries~\ref{cor: existence maximum good moduli space case} and~\ref{cor: existence maximum DM case}, in particular proving Proposition~\ref{thm: intro stabilization property}. More precisely, we get rid of the assumption of $\cX$ being of finite presentation over a quasi-separated locally-Noetherian algebraic space, and only require $\cX$ to be quasi-separated and locally-Noetherian itself. Moreover, in the adequate case we can replace $\DM$ with the weaker property $\cP_{\mathrm{sm}}$ of having smooth stabilizers. The argument is different and based on Lemma~\ref{lem: surjectivity stabilizers good and smooth stab case}.

\begin{lemma}\label{lem: nice chains are determined by relative inertia}
    Let $\cX\xrightarrow{\varphi}\cY\xrightarrow{\psi}\cZ$ be adequate moduli space morphisms between quasi-separated algebraic stacks. Suppose that at least one of the following holds:
    \begin{enumerate}
        \item $\varphi$ and $\psi$ are good moduli space morphisms, or
        \item $\cY$ and $\cZ$ have smooth geometric stabilizers.
    \end{enumerate}
    If the relative inertia $\cI_{\cX/\cY}$ and $\cI_{\cX/\cZ}$ coincide as subgroups of $\cI_{\cX}$, then $\psi$ is an isomorphism.
\end{lemma}
\begin{proof}
    The morphism $\cY\rightarrow\cZ$ is an isomorphism if and only if it is representable, which in turn is equivalent to $\cI_{\cY}\rightarrow\psi^*\cI_{\cZ}$ being a monomorphism. It is enough to check this at any geometric point, for instance by~\cite[Corollary 5.5.11]{Alp25}. By Lemma~\ref{lem: surjectivity stabilizers good and smooth stab case}, for every $y\in|\cY|$ with image $z\in|\cZ|$, if $x\in|\cX|$ is closed in the fiber over $y$ we have that the morphisms $G_{\overline{x}}\twoheadrightarrow G_{\overline{y}}\twoheadrightarrow G_{\overline{z}}$ between geometric stabilizers are all surjective. Therefore, $G_{\overline{y}}\rightarrow G_{\overline{z}}$ is injective if and only if $\ker(G_{\overline{x}}\rightarrow G_{\overline{y}})=\ker(G_{\overline{x}}\rightarrow G_{\overline{z}})$, which follows from $\cI_{\cX/\cY}=\cI_{\cX/\cZ}$.
\end{proof}

\begin{proposition}[Stabilization Property]\label{prop: generalization of stabilization}
    Let $\cX$ be a quasi-separated locally-Noetherian algebraic stack.
    Let $\cP$ be a property that implies $\Delta_{\mathrm{sep}}$, and at least one between
    $\cP_{\mathrm{sm}}$ or $\cP^{\mathrm{ca}}$. Suppose further that $\cX$ is Noetherian or $\cP$ is semi-local. Then, every ascending chain in $\cC_{\cX}^{\cP}$ admits an upper bound, and if $\cX$ is Noetherian, then every ascending chain in $\cC_{\cX}^{\cP}$ stabilizes. In particular, $\cC_{\cX}^{\cP}$ is either empty or admits a maximal element.
\end{proposition}
\begin{proof}
    By Lemma~\ref{lem: passing to quasicompact case principle}, we can reduce to the case where $\cX$ is Noetherian. Take an ascending chain in $\cC_{\cX}^{\cP}$, which in turn induces a descending chain of closed subgroups of $\cI_{\cX}$, by taking the relative inertia of the various adequate moduli space morphisms. As $\cX$ is Noetherian, this chain stabilizes. Then, the chain stabilizes by Lemma~\ref{lem: nice chains are determined by relative inertia}.
\end{proof}

\begin{corollary}\label{cor: generalization existence maximum}
    Let $\cX$ be a quasi-separated locally-Noetherian algebraic stack admitting an adequate moduli space $\cXMOD$, and $\cP$ a property modular with respect to $\cX$ that implies $\cP_{\mathrm{adm}}$ and at least one between $\cP^{\mathrm{ca}}$ and $\DM$.
    Suppose further that $\cX$ is Noetherian or $\cP$ is semi-local. Then, $\cC_{\cX}^{\cP}$ admits a maximum.
\end{corollary}
\begin{proof}
    By Proposition~\ref{prop: generalization of stabilization}, we know that $\cC_{\cX}^{\cP}$ admits a maximal element, which is unique by Lemma~\ref{lem: refinement property of CX}.
\end{proof}

\begin{remark}
    Since $\cP_{\mathrm{sm}}$ is not a modular property by Example~\ref{exm: non-modular properties}, in general there is no maximum in $\cC_{\cX}^{\cP_{\mathrm{sm}}}$, equivalently, the maximal element is not unique. For instance, consider the case where $\cX=B(\Gm\times\mu_2)$ over $\mathbb{F}_2$, and the two good moduli space morphisms $\varphi_1,\varphi_2:\cX\rightarrow B\Gm$ induced by $(\lambda,t)\mapsto\lambda$ and $(\lambda,t)\mapsto\lambda t$, respectively. Then, they are both maximal but non-isomorphic objects of $\cC_{\cX}^{\cP_{\mathrm{sm}}}$. Nevertheless, there is a morphism $B(\Gm\times\mu_2)\rightarrow B(\Gm\times\Gm)$ dominating both $\varphi_i$.
\end{remark}

We can refine Proposition~\ref{prop: generalization of stabilization} to generalize Lemma~\ref{lem: inclusion CX for good moduli spaces} and Corollary~\ref{cor: inclusion CX for DM case}.

\begin{proposition}\label{prop: generalization characterization CX}
    Let $\cX$ be a quasi-separated locally-Noetherian algebraic stack. Let $\cP$ be a property that implies $\cP_{\mathrm{adm}}$, and at least one between properties $\DM$ and $\cP^{\mathrm{ca}}$. Suppose further that $\cX$ is Noetherian or $\cP$ is semi-local. Then, $\cC_{\cX}^{\cP}$ is equivalent to a partially ordered subset of the set of subgroups of the inertia $\cI_{\cX}$, by sending an object $\varphi:\cX\rightarrow\cY$ to $\cC_{\cX}^{\cP}$ to the relative inertia $\cI_{\cX/\cY}\subset\cI_{\cX}$. Moreover, $\cI_{\cX/\cY}\subset\cI_{\cX}$ is closed (resp.\ open) if $\cP$ implies $\Delta_{\mathrm{sep}}$ (resp.\ $\DM$).
\end{proposition}
\begin{proof}
    Let $\varphi_1:\cX\rightarrow\cY_1$ and $\varphi_2:\cX\rightarrow\cY_2$ be two objects of $\cC_{\cX}^{\cP}$ with the same relative inertia; we want to prove that they are isomorphic. If there exists an arrow $\phi:\cY_1\rightarrow\cY_2$ between them in $\cC_{\cX}^{\cP}$, then $\phi$ is an isomorphism by Lemma~\ref{lem: nice chains are determined by relative inertia}. In general, consider the refinement $\varphi_{1,2}:\cX\rightarrow\cY_{1,2}$ of $\varphi_1$ and $\varphi_2$ as in Lemma~\ref{lem: refinement property of CX}. Notice that $\cY_{1,2}$ may not satisfy $\cP$ as we are not requiring $\cP$ to be a modular property; however $\cY_{1,2}$ and $\varphi_{1,2}$ satisfy at least one between $\DM$ and $\cP^{\mathrm{ca}}$. Since $\cY_{1,2}\rightarrow\cY_1\times_{\cXMOD}\cY_2$ is representable, we have $\cI_{\cX/\cY_{1,2}}=\cI_{\cX/\cY_1}\cap\cI_{\cX/\cY_2}=\cI_{\cX/\cY_1}=\cI_{\cX/\cY_2}$, and the statement follows from the previous case.
\end{proof}

Even though the results of this section already work under weak assumptions, it would be useful to get rid of the locally-Noetherian condition. However, one cannot expect chains in $\cC_{\cX}$ to stabilize in the non-Noetherian case, even when $\cX$ is of finite presentation over a quasi-compact scheme.

\section{Existence of Universal Morphisms and Applications}\label{subsec: examples and consequences}

Given an algebraic stack $\cX$ and a property $\cP$ of algebraic stacks, we consider the following fundamental question (Question~\ref{question: intro universal objects}): does there exist a morphism $\Phi_{\cP}:\cX\rightarrow\cX_{\cP}$ that is universal among morphisms from $\cX$ to stacks satisfying $\cP$?

When $\cP$ is the property of being an algebraic space, this is the same as asking about the existence of a moduli space for $\cX$.
We study Question~\ref{question: intro universal objects} when $\cX$ admits an adequate moduli space and $\cP$ is a modular property, using Corollary~\ref{cor: generalization existence maximum} as our main tool.

\subsection{The Main Existence Results}\label{subsec: existence of universal morphisms}

The first positive answer to Question~\ref{question: intro universal objects} is provided by the following theorem, for which we need to assume $\cP$ to be modular and contain $\cP_{\mathrm{mod}}$. This assumption will be later replaced with $\cP$ being strongly modular, that is sometimes more suitable for applications.

\begin{theorem}\label{thm: technical existence of initial morphisms}
    Let $\cX$ be a quasi-separated and locally-Noetherian algebraic stack admitting an adequate moduli space $\pi_{\cX}:\cX\rightarrow\cXMOD$. Let $\cP$ be a modular property of stacks that implies $\cP_{\mathrm{adm}}$.
    Assume that $\cX$ is Noetherian or $\cP$ is semi-local. Suppose that at least one of the following holds:
    \begin{enumerate}
        \item\label{case: technical good case thm main existence of initial morphisms} $\pi_{\cX}$ is a good moduli space morphism, or
        \item $\cP$ implies the property $\DM$.
    \end{enumerate}
    Then, the category of morphisms from $\cX$ to stacks having the property $\cP$ has an initial object
    \[
    \begin{tikzcd}
        \Phi_{\cP}:\cX\arrow[r] & \cX_{\cP}
    \end{tikzcd}
    \]
    that is an adequate moduli space morphism. Moreover, in case~\eqref{case: technical good case thm main existence of initial morphisms}, $\Phi_{\cP}$ is also a good moduli space morphism.
\end{theorem}

To prove the theorem, we need the following simple but technical lemma.

\begin{lemma}\label{lem: technical throw in lemma}
    Let $\cX$ be an algebraic stack with adequate moduli space $\pi_{\cX}:\cX\rightarrow\cXMOD$, and let $\cP$ be a modular property of stacks that implies $\cP_{\mathrm{adm}}$.
    Let $f:\cX\rightarrow\cZ$ be a morphism to an algebraic stack having property $\cP$. Then, there exists an object $\varphi:\cX\rightarrow\cY$ in $\cC_{\cX}^{\cP}$, a representable morphism $g:\cY\rightarrow\cZ$, and a 2-isomorphism $\alpha:g\circ\varphi\xrightarrow{\sim} f$, satisfying the following universal property:
    for any other triple $(\varphi',g',\alpha')$ where $\varphi':\cX\rightarrow\cY'$ is in $\cC_{\cX}^{\cP}$, $g':\cY'\rightarrow\cZ$ is a morphism, and $\alpha':g'\circ\varphi'\xrightarrow{\sim} f$ is a 2-isomorphism, there exists a morphism $\phi:\cY'\rightarrow\cY$, making the diagram
    \[
    \begin{tikzcd}[row sep=tiny]
        & \cY'\arrow[rd,"g'"]\arrow[dd,"\phi"]\\
        \cX\arrow[ru,"\varphi'"]\arrow[rd,"\varphi"'] & & \cZ\\
        & \cY\arrow[ru,"g"']
    \end{tikzcd}
    \]
    2-commutative, compatibly with $\alpha$ and $\alpha'$. Moreover, $\phi$ is unique up to a unique 2-isomorphism compatible with the other 2-isomorphisms.
\end{lemma}
\begin{proof}
    Since $\cP$ implies $\cP_{\mathrm{mod}}$, $\cZ$ admits an adequate moduli space, and we get a diagram of solid arrows
    \[
    \begin{tikzcd}
        \cX\arrow[rrd,bend right=20,"f"']\arrow[rr,bend left=30,"f'"]\arrow[r,dashed,"\varphi"'] & \cY\arrow[r,dashed]\arrow[rd,dashed,"g"] & \cZ'\arrow[r]\arrow[d] & \cXMOD\arrow[d]\\
        & & \cZ\arrow[r,"\pi_{\cZ}"] & \cZMOD
    \end{tikzcd}
    \]
    by~\cite[Theorem 3.12]{LocalStructureStacks_AHR25}. Set $\cY:=\underline{\Spec\!}_{\cZ'}(f'_*\cO_{\cX})$, completing the above diagram. Because $\cP$ is modular, also $\cY$ has $\cP$. As $\cP$ implies $\cP_{\mathrm{adm}}$, by Lemma~\ref{lem: cancellation properties}\eqref{lem: cancellation properties 2}, $\varphi$ is an adequate moduli space, and it is a good moduli space morphism if $\pi_{\cX}$ is. Therefore, it is enough to show that $\cY$ together with the induced representable morphism $g:\cY\rightarrow\cZ$ satisfies the universal property of the statement.

    Let $\cY'$, $\varphi'$ and $g'$ as in the statement of the lemma. The moduli space morphism $\cY'\rightarrow\cY'_{\mathrm{mod}}\simeq\cXMOD$ induces a map $h:\cY'\rightarrow\cZ'$ that factors $g'$ up to 2-isomorphism. Now, $h$ is quasi-compact and quasi-separated, hence it factors through $\phi:\cY'\rightarrow\underline{\Spec\!}_{\cZ'}(h_*\cO_{\cY'})\simeq\cY$. This yields the desired map $\phi$. The uniqueness of $\phi$ follows from~\cite[Lemma 7.23]{LocalStructureStacks_AHR25}.
\end{proof}

\begin{proof}[Proof of Theorem~\ref{thm: technical existence of initial morphisms}]
    By Corollary~\ref{cor: generalization existence maximum}, the category $\cC_{\cX}^{\cP}$ admits a maximum $\Phi_{\cP}:\cX\rightarrow\cX_{\cP}$, which is a good moduli space if $\pi_{\cX}$ is. We claim that this is the desired initial morphism.
    Let $f:\cX\rightarrow\cZ$ be a morphism to an algebraic stack with property $\cP$. Let $\varphi:\cX\rightarrow\cY$ be the object in $\cC_{\cX}^{\cP}$ with a morphism $g:\cY\rightarrow\cZ$ as in Lemma~\ref{lem: technical throw in lemma}. By maximality of $\cX_{\cP}$, there is a morphism $\phi:\cX_{\cP}\rightarrow\cY$, and the composite $g\circ\phi$ is the morphism we were looking for. The uniqueness follows from~\cite[Lemma 7.23]{LocalStructureStacks_AHR25}.
\end{proof}

It is not clear a priori whether the hypothesis that $\cP$ implies $\cP_{\mathrm{adm}}$ (and hence $\cP_{\mathrm{mod}}$) can be removed. The main hurdle is producing a morphism $g:\cY\rightarrow\cZ$ from an object $\cY$ of $\cC_{\cX}^{\cP}$. In the previous setting, such a morphism was provided by Lemma~\ref{lem: technical throw in lemma}. When $\cP$ is a strongly modular property, it is enough to find an object $\cY$ in $\cC_{\cX}$ and a representable morphism $g:\cY\rightarrow\cZ$, since $\cY$ will then automatically inherit $\cP$ by the definition of a strongly modular property (Definition~\ref{def: modular property}). The existence of such a morphism $g$ is guaranteed by the following theorem, which extends~\cite[Theorem 3.1]{AOV08b} and is of independent interest.

\begin{theorem}\label{thm: existence of relative moduli spaces}
    Let $\cX$ be an algebraic stack with an adequate moduli space $\pi_{\cX}:\cX\rightarrow\cXMOD$. Let $f:\cX\rightarrow\cZ$ be a morphism of algebraic stacks, with $\cZ$ having Zariski-locally affine diagonal.
    Then, there exists an algebraic stack $\cY$, morphisms $\cX\xrightarrow{\varphi}\cY\xrightarrow{g}\cZ$, and a 2-isomorphism $\alpha:g\circ\varphi\xrightarrow{\sim} f$ such that:
    \begin{enumerate}
        \item $g$ is representable,
        \item\label{point: universality property existence of relative moduli spaces} for any other 2-factorization $\cX\xrightarrow{\varphi'}\cY'\xrightarrow{g'}\cZ$, $\alpha':g'\circ\varphi'\xrightarrow{\sim} f$, with $g'$ representable, there exists a unique morphism $\phi:\cY\rightarrow\cY'$ and 2-morphisms $\lambda:\phi\circ\varphi\xrightarrow{\sim}\varphi'$, $\tau:g'\circ\phi\xrightarrow{\sim} g$ making the following diagram commute
        \[
        \begin{tikzcd}
            g'\circ\phi\circ\varphi\arrow[r,"\tau\circ\varphi"]\arrow[d,"g'(\lambda)"] & g\circ\varphi\arrow[d,"\alpha"]\\
            g'\circ\varphi'\arrow[r,"\alpha'"] & f
        \end{tikzcd}
        \]
        \item $\varphi$ is an adequate moduli morphism,
        \item the construction commutes with base change along representable flat morphisms with target either $\cY$ or $\cZ$.
    \end{enumerate}
\end{theorem}
\begin{proof}
    The proof is similar to the one of~\cite[Theorem 3.1]{AOV08b}. First, by the claimed commutativity of the construction with respect to representable flat morphisms with target $\cZ$, we can assume that $\cZ$ has affine diagonal. Since $\cZ$ has affine diagonal, there exists an affine smooth presentation $U\rightarrow\cZ$ from an algebraic space. Let $R:=U\times_{\cZ}U\rightrightarrows U$ be the corresponding smooth and affine groupoid. Let $\cX_U:=\cX\times_{\cZ}U$ and $\cX_R:=\cX\times_{\cZ}R$, that has again a groupoid structure. Since $\cX_U\rightarrow\cX$ is affine, by~\cite[Lemma 5.2.11]{AdequateModAlp14} $\cX_U$ admits an adequate moduli spaces $(\cX_U)_\mathrm{mod}$ and similarly for $\cX_{R}$. By the universal property of adequate moduli space~\cite[Theorem 3.12]{LocalStructureStacks_AHR25}, we obtain a groupoid $(\cX_R)_\mathrm{mod}\rightrightarrows(\cX_U)_\mathrm{mod}$; we claim that the morphisms are smooth and affine. Let $p:R\rightarrow U$ be any of the two projections, and set $V:=(\cX_U)_\mathrm{mod}\times_{U}R$, which is smooth and affine over $(\cX_U)_\mathrm{mod}$. Since adequate moduli spaces commute with flat base change by~\cite[Proposition 5.2.9 (1)]{AdequateModAlp14}, the morphism $\cX_R\simeq \cX_U\times_{(\cX_U)_\mathrm{mod}}V\rightarrow V$ is an adequate moduli space. By uniqueness of adequate moduli spaces, we get that $(\cX_R)_\mathrm{mod}\simeq(\cX_U)_\mathrm{mod}\times_{U}R$. Applying the same reasoning to the other projection $R\rightarrow U$, we showed that $(\cX_R)_\mathrm{mod}\rightrightarrows(\cX_U)_\mathrm{mod}$ is a smooth groupoid. The following diagram summarizes the situation
    \[
    \begin{tikzcd}
        \cX_R\arrow[d,shift right]\arrow[d,shift left]\arrow[r] & (\cX_R)_\mathrm{mod}\arrow[d,shift right]\arrow[d,shift left]\arrow[r] & R\arrow[d,shift right]\arrow[d,shift left]\\
        \cX_U\arrow[d]\arrow[r] & (\cX_U)_\mathrm{mod}\arrow[d]\arrow[r] & U\arrow[d]\\
        \cX\arrow[r,"\varphi"] & \cY\arrow[r,"g"] & \cZ
    \end{tikzcd}
    \]
    where $\cY$ is the quotient of the groupoid in the middle. Then, there are induced morphisms $\varphi$ and $g$ as above making all the square diagrams 2-cartesian. By descent, $\varphi$ is an adequate moduli space~\cite[Proposition 5.2.9 (2)]{AdequateModAlp14}, and $g$ is representable. Moreover, the construction commutes with representable flat base change along morphisms to either $\cZ$ or $\cY$, since the formation of adequate moduli spaces commute with flat morphisms by~\cite[Proposition 5.2.9 (1)]{AdequateModAlp14}. Therefore, it is enough to prove the universal property of $\cY$.

    Let $g':\cY'\rightarrow\cZ$ be as in point~\eqref{point: universality property existence of relative moduli spaces}, and form $\cY'_U$ and $\cY'_R$ by pulling back $U$ and $R$ along $\cY\rightarrow\cZ$, respectively. Since $g'$ is representable, these are algebraic spaces, and the universal property of adequate moduli spaces~\cite[Theorem 3.12]{LocalStructureStacks_AHR25} yields morphisms $(\cX_U)_\mathrm{mod}\rightarrow\cY'_U$ and $(\cX_R)_\mathrm{mod}\rightarrow\cY'_R$. This gives a cartesian diagram of groupoids, hence an induced morphism $\varphi:\cY\rightarrow\cY'$ between quotients. A simple check shows that this is the desired morphism.
\end{proof}

The following theorem replaces the assumption in Theorem~\ref{thm: technical existence of initial morphisms} of $\cP$ containing $\cP_{\mathrm{mod}}$ with $\cP$ being strongly modular.

\begin{theorem}\label{thm: existence of universal morphism strongly modular case}
    Let $\cX$ be a quasi-separated locally-Noetherian algebraic stack admitting an adequate moduli space $\pi_{\cX}:\cX\rightarrow\cXMOD$. Let $\cP$ be a strongly modular property of stacks that implies $\Delta_{\mathrm{aff}}^{\mathrm{loc}}$, and assume that $\cX$ is Noetherian or $\cP$ is semi-local. Suppose that at least one of the following holds:
    \begin{enumerate}
        \item\label{case: good case thm main existence of initial morphisms} $\pi_{\cX}$ is a good moduli space morphism, or
        \item $\cP$ implies the property $\DM$.
    \end{enumerate}
    Then, the category of morphisms from $\cX$ to stacks having the property $\cP$ has an initial object
    \[
    \begin{tikzcd}
        \Phi_{\cP}:\cX\arrow[r] & \cX_{\cP}
    \end{tikzcd}
    \]
    that is an adequate moduli space morphism. Moreover, in case~\eqref{case: good case thm main existence of initial morphisms}, $\Phi_{\cP}$ is also a good moduli space morphism.
\end{theorem}
\begin{proof}
    Let $\Phi_{\cP}:\cX\rightarrow\cX_{\cP}$ be the maximum in $\cC_{\cX}^{\cP}$, which exists by Corollary~\ref{cor: generalization existence maximum}.
    By Theorem~\ref{thm: existence of relative moduli spaces}, there exists an adequate moduli space morphism $\varphi:\cX\rightarrow\cY$ and a morphism $g:\cY\rightarrow\cZ$ that factors $f$ up to 2-isomorphism. Since $\cP$ is strongly modular and $\cZ$ satisfies $\cP$, also $\cY$ is an object of $\cC_{\cX}^{\cP}$. This yields a morphism $\cX_{\cP}\rightarrow\cY$, whose composite with $g$ gives the desired factorization of $f$. The uniqueness of the factorization follows from~\cite[Lemma 7.23]{LocalStructureStacks_AHR25}.
\end{proof}

Now, we apply the results above to some concrete cases, when $\cP$ is one of the properties in Example~\ref{exm: modular properties}.
Theorem~\ref{thm: technical existence of initial morphisms} can be directly applied to get an initial morphism to stacks with finite inertia.

\begin{corollary}\label{cor: finite inertia initial map}
    Let $\cX$ be a quasi-separated locally-Noetherian algebraic stack admitting a good moduli space $\pi_{\cX}:\cX\rightarrow\cXMOD$. Then, there exists a good moduli space morphism $\Phi_{\mathrm{fi}}:\cX\rightarrow\cX_{\mathrm{fi}}$ to an algebraic stack with finite inertia, that is initial in the category of morphisms from $\cX$ to algebraic stacks with finite inertia and whose morphism to their moduli space has affine diagonal.
\end{corollary}
\begin{proof}
This follows from Theorem~\ref{thm: technical existence of initial morphisms}, as every quasi-separated algebraic stack with finite inertia admits an adequate moduli space by~\cite{GeneralKeelMori_Ryd07} and~\cite[Theorem 8.3.2]{AdequateModAlp14}.
Notice that $\cI_{\mathrm{fi}}$ is a semi-local property.
\end{proof}

Theorem~\ref{thm: existence of universal morphism strongly modular case} instead yields an initial morphism to Deligne-Mumford stacks, thus proving Theorem~\ref{thm: intro existence XDM}.

\begin{corollary}\label{cor: DM-fication}
    Let $\cX$ be a quasi-separated locally-Noetherian algebraic stack admitting an adequate moduli space $\pi_{\cX}:\cX\rightarrow\cXMOD$.
    Then, there exists an adequate moduli space morphism $\Phi_{\DM}:\cX\rightarrow\cX_{\DM}$ to a Deligne-Mumford stack with Zariski-locally affine diagonal, that is initial in the category of morphisms from $\cX$ to Deligne-Mumford stacks with Zariski-locally affine diagonal. Moreover, $\cXDM$ has finite inertia.
\end{corollary}
\begin{proof}
    Theorem~\ref{thm: existence of universal morphism strongly modular case} yields the initial morphism $\PhiDM:\cX\rightarrow\cXDM$, as $\DM$ is clearly semi-local. The finiteness of the inertia follows from the fact that $\cXDM$ admits an adequate moduli space by Lemma~\ref{lem: cancellation properties}\eqref{lem: cancellation properties 3}, the diagonal is quasi-finite and separated, and from~\cite[Theorem 8.3.2]{AdequateModAlp14}.
\end{proof}

There are other properties of algebraic stacks which are interesting to apply Theorems~\ref{thm: technical existence of initial morphisms} and~\ref{thm: existence of universal morphism strongly modular case} to, like the property $\cP_{\mathrm{gq}}$ of being a global quotient, which we do not pursue in this paper. The property $\UDM$ of being uniformizable is studied in the next subsection, where we give a more explicit construction and relate it to the fundamental group of $\cX$.

\subsection{Uniformizable Stacks and Fundamental Group}\label{subsec: XUDM}

Recall that an algebraic stack $\cX$ is \emph{uniformizable} if it admits a finite, étale presentation by an algebraic space. In that case, $\cX$ is necessarily Deligne-Mumford. When $\cP=\UDM$, we can reinterpret Theorem~\ref{thm: existence of universal morphism strongly modular case} in terms of the \emph{étale fundamental group} of $\cX$. We refer to~\cite{Noo04,landi2025stacksmonodromysymmetriccubic} for the definition and basic facts on fundamental groups. We note that in~\cite{Noo04}, all stacks are implicitly assumed to be quasi-separated, as the conventions of~\cite{champsalgebrique} are adopted, contrarily to us.

To treat the locally-Noetherian case, let us introduce the following definition, which is a special case of Remark~\ref{rmk: semi-localizing a property}.
\begin{definition}\label{def: semi-local uniformizability}
   We say that an algebraic stack $\cX$ is \emph{semi-locally uniformizable} if every quasi-compact open subscheme of $\cX$ is uniformizable.
\end{definition}

We start by proving that $\Phi_{\UDM}$ induces an isomorphism between fundamental groups.

\begin{proposition}\label{prop: XUDM isomorphism of fundamental groups}
    Let $\cX$ be a quasi-separated locally-Noetherian algebraic stack admitting an adequate moduli space $\pi_{\cX}:\cX\rightarrow\cXMOD$. Let $\Phi_{\UDM}:\cX\rightarrow\cXUDM$ be the initial morphism to a semi-locally uniformizable algebraic stack. Then, $\Phi_{\UDM}$ induces an isomorphism between fundamental groups at any connected component. 
\end{proposition}
\begin{proof}
    We can assume $\cX$ to be connected. Let $\cU\rightarrow\cX$ be a finite étale cover, with Galois closure $\cY\xrightarrow{\varphi}\cU\rightarrow\cX$. Let $\psi$ be the composite of the two morphisms, which is a torsor under a finite group $G$, while $\varphi$ is an $H$-torsor for some subgroup $H\subset G$. As $\psi$ is affine, by~\cite[Lemma 5.2.11]{AdequateModAlp14} we can construct a commutative diagram
    \[
    \begin{tikzcd}
        \cY\arrow[d]\arrow[r,"\varphi"] & \cU\arrow[d]\arrow[r] & \cX\arrow[d,"f"]\\
        \cYMOD\arrow[r] & {[\cYMOD/H]}\arrow[r] & {[\cYMOD/G]}
    \end{tikzcd}
    \]
    with cartesian squares. By the universal property of $\cXUDM$, the morphism $f:\cX\rightarrow[\cYMOD/H]$ factors through $\cXUDM$. Then, $[\cYMOD/H]\times_{[\cYMOD/G]}\cXUDM$ is a finite étale cover over $\cXUDM$ whose pullback to $\cX$ is $\cU\rightarrow\cX$. This shows that the morphism between fundamental groups is injective. By~\cite[Tag 0BN6]{Sta24}, the surjectivity follows from the fact that the target of an adequate moduli space morphism is connected if and only if the source is~(\cite[Theorem 5.3.1]{AdequateModAlp14}), and that $\PhiUDM$ is an adequate moduli space morphism by Theorem~\ref{thm: existence of universal morphism strongly modular case}.
\end{proof}

Proposition~\ref{prop: XUDM isomorphism of fundamental groups} is useful to reduce problems about fundamental groups of algebraic stacks to the simpler case of uniformizable stacks. See for instance the later Remark~\ref{rmk: generalization Noohi}.

The above discussion and Proposition~\ref{prop: XUDM isomorphism of fundamental groups} also suggest that we can use Noohi's theory of fundamental groups to give an alternative construction of $\cXUDM$, which we will do next. First, let us recall some results of~\cite{Noo04} that we will need.
When $\cX$ is Noetherian and Deligne-Mumford, Noohi showed that $\cX$ is uniformizable if and only if for every geometric point $x$ of $\cX$ with residual gerbe $\cG_x$, the pushforward
\[
\begin{tikzcd}
    \omega_x:\pi_1(\cG_x,x)\arrow[r] & \pi_1(\cX,x)
\end{tikzcd}
\]
between fundamental groups is injective. Moreover, when $\cX$ is connected and admits a moduli space $\cXMOD$ in the sense of~\cite[Definition 7.1]{Noo04}, the surjective morphism $\pi_{\cX*}:\pi_1(\cX,x)\rightarrow\pi_1(\cXMOD,\pi_{\cX}(x))$ has kernel equal to the closed normal subgroup $N$ generated by the images of all $\omega_{x'}$, see~\cite[Theorem 7.11]{Noo04}.

\begin{remark}\label{rmk: generalization Noohi}
In~\cite[Definition 7.1]{Noo04}, the author requires the moduli space map to induce a bijection between geometric points (up to 2-isomorphism), which is too restrictive in general. Proposition~\ref{prop: XUDM isomorphism of fundamental groups} reduces statements about fundamental groups to the simpler case of uniformizable stacks, where~\cite[Theorem 7.11]{Noo04} applies directly. Alternatively, one can generalize Noohi's argument to prove the result for adequate moduli spaces, using the more general version~\cite[Theorem 3.14]{LocalStructureStacks_AHR25} of Luna's fundamental lemma.
\end{remark}

Now, we present an alternative construction of $\cXUDM$ using the results of~\cite{Noo04}.

\begin{construction}\label{construction XUDMvarphi}
    Let $\cX$ be a connected Noetherian algebraic stack that admits an adequate moduli space $\pi_{\cX}:\cX\rightarrow\cXMOD$. By~\cite[Theorem 11.4]{Noo04}, there exists a finite étale cover $\varphi:\cU\rightarrow\cX$ from a connected algebraic stack such that $\omega_u$ is trivial for every geometric point $u$ of $\cU$. The same property is satisfied by $\cV$ for any connected finite étale cover $\cV\rightarrow\cU$. Therefore, eventually after passing to the Galois closure, we can assume $\cU\rightarrow\cX$ to be a connected Galois cover under a finite group $G$. Since $\cU\rightarrow\cX$ is affine, by~\cite[Lemma 5.2.11]{AdequateModAlp14} $\cU$ admits an adequate moduli space $\pi_{\cU}:\cU\rightarrow\cUMOD$. The action of $G$ on $\cU$ induces an action of $G$ on $\cUMOD$, and we define $\cXUDM^{\varphi}:=[\cUMOD/G]$. Since the composite $\cU\rightarrow\cUMOD\rightarrow[\cUMOD/G]$ is $G$-invariant by construction, it induces a morphism $\Phi^{\varphi}_{\mathrm{UDM}}:\cX\rightarrow\cXUDM^{\varphi}$ making the diagram
    \[
    \begin{tikzcd}[column sep=large, row sep=small]
        \cU\arrow[d,"\varphi"']\arrow[r] & \cUMOD\arrow[d]\\
        \cX\arrow[r,"\Phi_{\UDM}^{\varphi}"] & \cX_{\UDM}^{\varphi}
    \end{tikzcd}
    \]
    cartesian. By~\cite[Theorem 7.11]{Noo04} and Remark~\ref{rmk: generalization Noohi}, the pushforward $\pi_{\cU*}:\pi_1^{\text{ét}}(\cU,u)\rightarrow\pi_1^{\text{ét}}(\cUMOD,\Phi(u))$ is an isomorphism for every geometric point $u$ of $\cU$.
\end{construction}

\begin{lemma}\label{lem: basic properties of XUDM}
    For any $\varphi$ as in Construction~\ref{construction XUDMvarphi}, $\Phi^{\varphi}_{\UDM}:\cX\rightarrow\cXUDM^{\varphi}$ is an adequate moduli space morphism and it induces an isomorphism between fundamental groups.
\end{lemma}
\begin{proof}
    The first claim follows by flat descent. For the second, let $x$ and $u$ be geometric points of $\cX$ and $U$ respectively, with $\varphi(u)=x$. By the theory of fundamental groups~\cite{Noo04}, we have a commutative diagram
    \[
	\begin{tikzcd}
		0\arrow[r] & \pi_1^{\text{ét}}(\cU,u)\arrow[r]\arrow[d,"\pi_{\cU*}","\simeq"'] & \pi_1^{\text{ét}}(\cX,x)\arrow[r]\arrow[d] & G\arrow[d,"="]\arrow[r] & 0\\
		0\arrow[r] & \pi_1^{\text{ét}}(\cUMOD,\pi_{\cU}(u))\arrow[r] & \pi_1^{\text{ét}}(\cX^{\varphi}_{\UDM},\Phi^{\varphi}_{\UDM}(x))\arrow[r] & G\arrow[r] & 0
	\end{tikzcd}
	\]
    with cartesian squares, where the exactness of the bottom row follows from the fact that $\cUMOD$ is connected.
    The statement follows from the five lemma.
\end{proof}

Notice that the algebraic stack $\cX^{\varphi}_{\UDM}$ is uniformizable for every $\varphi$ as in Construction~\ref{construction XUDMvarphi}.
The next proposition shows that $\Phi_{\UDM}^{\varphi}$ is initial among morphisms to uniformizable stacks, thus recovering $\cXUDM$.

\begin{proposition}\label{prop: universality XUDMvarphi}
 	Let $\varphi$ be as in Construction~\ref{construction XUDMvarphi}. The morphism $\Phi^{\varphi}_{\UDM}:\cX\rightarrow\cX^{\varphi}_{\UDM}$ is initial in the category of morphisms from $\cX$ to uniformizable algebraic stacks. In particular, $\cXUDM^{\varphi}\simeq\cXUDM$ and it does not depend on $\varphi$.
 \end{proposition}
 \begin{proof}
 	Let $f:\cX\rightarrow\cZ$ be a morphism to a uniformizable algebraic stack; we want to construct a factorization of $f$ through $\Phi^{\varphi}_{\UDM}$. For this, we can assume $\cZ$ to be connected and Noetherian, as $\cX$ is. Let $\pi:Z\rightarrow\cZ$ be a uniformization that is Galois under a group $H$. Let $\cV:=\cX\times_{\cZ}Z$, whose projection $\cV\rightarrow\cX$ is then an $H$-torsor $\cX$. By Lemma~\ref{lem: basic properties of XUDM} and the theory of fundamental groups, there exists an $H$-torsor $\cV^{\varphi}\rightarrow\cX^{\varphi}_{\UDM}$ and a morphism $\cV\rightarrow\cV^{\varphi}$ making the following diagram of solid arrows
    \begin{equation}\label{diag: universal property XUDMvarphi}
    \begin{tikzcd}
        \cV\arrow[d]\arrow[r]\arrow[rr,bend left=20] & \cV^{\varphi}\arrow[d]\arrow[r,dashed] & Z\arrow[d]\\
        \cX\arrow[rr,bend right=20]\arrow[r,"\Phi^{\varphi}_{\UDM}"] & \cX^{\varphi}_{\UDM}\arrow[r,dashed] & \cZ
    \end{tikzcd}
    \end{equation}
    commutative, with cartesian square. By construction, also the outer rectangle is cartesian. Again by Lemma~\ref{lem: basic properties of XUDM}, $\Phi^{\varphi}_{\UDM}$ is an adequate moduli space, thus the same holds for $\cV\rightarrow\cV^{\varphi}$. By the universal property of adequate moduli spaces~\cite[Theorem 3.12]{LocalStructureStacks_AHR25}, this induces a morphism $\cV^{\varphi}\rightarrow Z$, whose composite with $Z\rightarrow\cZ$ is $H$-invariant. This in turn induces a morphism $\cXUDM^{\varphi}\rightarrow\cZ$ completing~\eqref{diag: universal property XUDMvarphi} to a commutative diagram with cartesian squares.
    The uniqueness of the factorization follows from~\cite[Lemma 7.23]{LocalStructureStacks_AHR25}.
 \end{proof}
 \begin{remark}\label{rmk: characterization of XUDM in terms of fundamental group}
     The proof above shows that $\PhiUDM$ can be characterized as the adequate moduli space morphism with target a uniformizable stack and such that it induces an isomorphism between étale fundamental groups.
 \end{remark}

\subsection{On Base Change Properties}\label{subsec: base change properties}

In general, the formation of initial morphisms $\pi_{\cX}:\cX\rightarrow\cX_{\cP}$ does not commute with base change, even along open immersions or smooth surjective morphisms. In this section, we explore under which assumptions the base change property holds, for $\cP$ between $\UDM$, $\DM$, and $\cI_{\mathrm{fi}}$.

\subsubsection{Base Change Property of $\cXUDM$}
We specialize the discussion to $\cP=\UDM$. Unfortunately, the construction of $\PhiUDM:\cX\rightarrow\cXUDM$ does not commute with base change, even along open immersions, as the following example shows.

\begin{example}\label{exm: root XUDM}
	Let $\cX$ be the $\mu_2$-root gerbe over $\P^1$ associated to the line bundle $\cO_{\P^1}(1)$; this is the only non-trivial $\mu_2$-gerbe over $\P^1$. Then, $\cX$ is simply connected by~\cite[Example 5.10, Example 9.2]{Noo04}, in particular $\cXUDM\simeq\cXMOD\simeq\P^1$. Consider any open immersion $\A^1\subset\P^1$. As the restriction of $\cO_{\P^1}(1)$ to $\A^1$ is trivial, the restriction of the root gerbe $\cX\rightarrow\P^1$ over $\A^1$ is $\cY:=\A^1\times B\mu_2$. On the other hand, $\cY_{\mathrm{UDM}}\simeq\cY$, in particular $\cY_{\mathrm{UDM}}\not\simeq\cXUDM|_{\A^1}$.
	This phenomenon is typical of root gerbes and weighted projective stacks.
\end{example}

Nevertheless, the construction behaves well with respect to finite étale morphisms, as predictable.

\begin{proposition}\label{prop: base change property of XUDM}
	Let $\cX$ be a quasi-separated locally-Noetherian algebraic stack admitting an adequate moduli space, and let $\Phi_{\UDM}:\cX\rightarrow\cXUDM$ be the initial morphism to uniformizable stacks. Let $\cZ\rightarrow\cXUDM$ be a finite, étale morphism, and let $\cY:=\cX\times_{\cXUDM}\cZ$ with projection $\psi:\cY\rightarrow\cZ$. Then, $\psi$ is initial in the category of morphisms from $\cY$ to uniformizable stacks. In other words, the formation of $\PhiUDM:\cX\rightarrow\cXUDM$ commutes with base change along finite, étale morphisms to $\cXUDM$.
\end{proposition}
\begin{proof}
	We can assume both $\cX$ and $\cZ$ to be connected, so the same holds for $\cXUDM$ and $\cY$. Suppose given a connected uniformizable stack $\cU$ and a morphism $f:\cY\rightarrow\cU$. Let $U\rightarrow\cU$ be a $G$-torsor from a connected algebraic space, and let $\cV:=\cY\times_{\cU}U\rightarrow\cY$ be the pullback. Since the composite of $\cV\rightarrow\cY\rightarrow\cX$ is finite and étale, and the fundamental groups of $\cX$ and $\cXUDM$ are isomorphic by Proposition~\ref{prop: XUDM isomorphism of fundamental groups}, there exists a finite étale cover $\overline{\cV}\rightarrow\cXUDM$ such that $\cV\simeq\cX\times_{\cXUDM}\overline{\cV}$. Let $\cV_1:=\cZ\times_{\cXUDM}\overline{\cV}$. Then, $\cY\times_{\cZ}\cV_1\simeq\cY\times_{\cX}\cV\rightarrow\cV$ is finite and étale, and admits a section $s$. The composite of $s$ with the projection to $\cV_1$ has image a connected component $\overline{\cU}$ of $\cV_1$. Since the formation of adequate moduli spaces commutes with flat morphisms, the resulting morphism $\cV\rightarrow\overline{\cU}$ is an adequate moduli space morphism such that $\cV\simeq\cY\times_{\cZ}\overline{\cU}$. In particular, $\overline{\cU}\rightarrow\cZ$ is a $G$-torsor, and there is an induced morphism $\overline{\cU}\rightarrow U$ that is $G$-equivariant. This yields a morphism $\cZ\rightarrow\cU$ between quotients, as wanted. The uniqueness up to unique 2-isomorphism follows from~\cite[Lemma 7.23]{LocalStructureStacks_AHR25}.
\end{proof}

\subsubsection{Base Change Properties of $\cXDM$ and $\cX_{\mathrm{fi}}$}

In general, the same problem as for $\cXUDM$ is present for $\cP=\DM$ or $\cP=\cI_{\mathrm{fi}}$, as the following example suggested by David Rydh shows.

\begin{example}\label{exm: Rydh counterexample for DM and fi}
    Let $k$ be a field of characteristic different from 2. Consider the action of $\Gm$ on $\A^2$ with weights $(2,-2)$, and set $\cX=[\A^2/\Gm]$. Then, $\cX_{\mathrm{fi}}\simeq\cXDM\simeq\cXMOD\simeq\A^1$. Let $\cU=\cX\times_{\cXMOD}(\A^1\setminus0)\simeq[(\A^2\setminus0)/\Gm]$. Then, $\cU$ is already a DM-stack with finite inertia, even uniformizable. See Example~\ref{exm: counterexample non-commutativity finite inertia} for the characteristic 2 case.
\end{example}

We characterize when the formation of $\cXDM$ and $\cX_{\mathrm{fi}}$ commutes with base change, under the assumption that $\pi_{\cX}:\cX\rightarrow\cXMOD$ is a good moduli space. For simplicity, we will work over an algebraically closed field. In this setting, if the base change property holds then it is also possible to give a local presentation of $\cXDM$ and $\cX_{\mathrm{fi}}$.

For every geometric point $x$ of $\cX$, we denote by $G_x$ its stabilizer group scheme, and by $G_x^{\circ}$ its connected component of the identity. By~\cite[Tag 0B7R]{Sta24}, $G_x^{\circ}$ is an open and closed characteristic subgroup scheme.

\begin{proposition}\label{prop: characterization commutation with base change DM case}
    Let $\cX$ be an algebraic stack with affine diagonal and of finite type over an algebraically closed field $k$. Assume that it admits a good moduli space $\pi_{\cX}:\cX\rightarrow\cXMOD$, and let $\PhiDM:\cX\rightarrow\cXDM$ be the initial morphism from $\cX$ to a Deligne-Mumford stack with affine diagonal. The following are equivalent:
    \begin{enumerate}
        \item\label{point: arbitrary base change} $\PhiDM$ commutes with base change along morphisms with affine diagonal from Deligne-Mumford stacks to $\cXDM$;
        \item\label{point: étale base change} $\PhiDM$ commutes with base change along affine étale morphisms to $\cXMOD$;
        \item\label{point: simultaneous local structure} for every point $x\in\cX(k)$, there exists an affine étale neighborhood $U\rightarrow\cXMOD$ of $\pi_{\cX}(x)$ such that $\cX\times_{\cXMOD}U\simeq[\Spec A/G_x]$ and $\cXDM\times_{\cXMOD}U\simeq[\Spec A^{G_x^{\circ}}/(G_x/G_x^{\circ})]$;
        \item\label{point: characterization stabilizers} for every point $x\in\cX(k)$, with image $y:=\Phi_{\DM}(x)$ in $\cXDM$, the morphism $G_x\rightarrow G_y$ induces an isomorphism $G_x/G_x^{\circ}\simeq G_y$;
        \item\label{point: nice local structure} for every point $x\in\cX(k)$, there exists an affine étale neighborhood $U\rightarrow\cXMOD$ of $\pi_{\cX}(x)$ such that $\cX\times_{\cXMOD}U\simeq[\Spec A/G_x]$ and $G_{x'}^{\circ}=G_{x'}\cap G_x^{\circ}$ for any other closed point $x'$ in $\cX\times_{\cXMOD}U$.
    \end{enumerate}
\end{proposition}
\begin{proof}
    First, recall that $\cXDM$ exists by Theorem~\ref{thm: existence of universal morphism strongly modular case}, and it has affine diagonal (Lemma~\ref{lem: inheriting affine diagonal}) and finite inertia.
    Clearly,~\eqref{point: arbitrary base change} implies~\eqref{point: étale base change}. Now, we show that~\eqref{point: étale base change} implies~\eqref{point: simultaneous local structure}. By~\cite[Theorem 4.12]{AHR20} and the assumption that the formation of $\cXDM$ commutes with base change along affine étale morphisms to $\cXMOD$, we can assume that $\cX\simeq[\Spec A/G_x]$, $G_x=\underline{\mathrm{Aut}}_{\cX}(x)$. By definition, the natural morphism $f:\cX\rightarrow[\Spec A^{G_x^{\circ}}/(G_x/G_x^{\circ})]=:\cZ$ factors as $\cX\rightarrow\cXDM\xrightarrow{\varphi}\cZ$. By construction and Lemma~\ref{lem: surjectivity stabilizers good and smooth stab case}, $\varphi$ induces an isomorphism between stabilizers at $\PhiDM(x)$ and $f(x)$. Since both $f$ and $\PhiDM$ are good moduli space morphisms, also $\varphi$ is by Lemma~\ref{lem: cancellation properties}\eqref{lem: cancellation properties 3}. Therefore, $\varphi$ is an isomorphism in a neighborhood of $f(x)$. Since $\cZ\rightarrow\cXMOD$ is a universal homeomorphism, $\varphi$ is an isomorphism over a neighborhood of $\pi_{\cX}(x)$ in $\cXMOD$. Restricting to that open yields the presentation in point~\eqref{point: simultaneous local structure}. Notice that~\eqref{point: simultaneous local structure} immediately implies~\eqref{point: characterization stabilizers}.
    
    Now, we show that point~\eqref{point: characterization stabilizers} implies the arbitrary base change property in~\eqref{point: arbitrary base change}. First, notice that the isomorphism $G_y\simeq G_x/G_x^{\circ}$ holds for all geometric points $x$ in $\cX$, $y=\PhiDM(x)$. Let $\cY\rightarrow\cXDM$ be a morphism with affine diagonal from a Deligne-Mumford stack, and let $\cU:=\cX\times_{\cXDM}\cY$. Let $f:\cU\rightarrow\cZ$ be a morphism to a Deligne-Mumford stack with affine diagonal; we need to show that $f$ factors through $\cY$. Since $\cZ$ is a Deligne-Mumford stack, $\cI_{\cU/\cZ}$ is an open substack of $\cI_{\cU}$ whose fiber over a geometric point $u$ in $\cU$ contains the connected component $G_u^{\circ}$ of the stabilizer group of $u$. If $x$ is the image of $u$ in $\cX$, $x_0:=\PhiDM(x)$, and $y$ is the image of $u$ in $\cY$, we have $G_u\simeq G_x\times_{G_{x_0}}G_y$. Since $\cY$ is Deligne-Mumford, we also have $G_u^{\circ}\simeq G_x^{\circ}\times_{G_{x_0}}G_y$, hence $G_u/G_u^{\circ}\simeq G_y$. It follows that $\cI_{\cU/\cY}\subset\cI_{\cU/\cZ}$. As $\cU\rightarrow\cY$ is also of finite presentation,~\cite[Theorem 7.22]{LocalStructureStacks_AHR25} yields the desired factorization.
    We have proved that points~\eqref{point: arbitrary base change},~\eqref{point: étale base change},~\eqref{point: simultaneous local structure} and~\eqref{point: characterization stabilizers} are all equivalent.
    
    Now, we show that~\eqref{point: simultaneous local structure} implies point~\eqref{point: nice local structure}. Let $x'$ be another closed point of $\cX$, which we assume to be as in point~\eqref{point: simultaneous local structure}. Notice that $G_{x'}^{\circ}\subset G_{x'}\cap G_x^{\circ}$ is always true. By Lemma~\ref{lem: surjectivity stabilizers good and smooth stab case} and diagram chasing, the equality holds if and only if the induced morphism $G_{x'}/G_{x'}^{\circ}\rightarrow G_x/G_x^{\circ}$ is injective, which is automatic given the presentation. To conclude the proof of the proposition, we show that~\eqref{point: nice local structure} implies point~\eqref{point: characterization stabilizers}. Let $U\rightarrow\cXMOD$ as in~\eqref{point: nice local structure}, and let $\cU:=\cX\times_{\cXMOD}U$.
    The good moduli space morphism $\cU\rightarrow[\Spec A^{G_x^{\circ}}/(G_x/G_x^{\circ})]$ has open and closed relative inertia subgroup $\cJ_U\subset\cI_{\cU}$. By construction and our assumption, if $u'$ is a point over $x'$, we have $(\cJ_U)_{u'}=G_{x'}\cap G_x^{\circ}=G_{x'}^{\circ}$. Moreover, $\cI_{\cU}\simeq\cI_{\cX}\times_{\cX}\cU$, hence the image $\widetilde{\cJ}_U$ of $\cJ_U$ in $\cI_{\cX}$ its open. Notice that $\cJ_U\subset\widetilde{\cJ}_U\times_{\cX}\cU$ are both open substacks of $\cI_{\cU}$, locally of finite type over $k$ and with the same closed points, hence they coincide. Denote by $\cI_{\cX}^{\circ}$ the union of all $\widetilde{\cJ}_U$ with $U\rightarrow\cXMOD$ varying among affine étale morphisms as in~\eqref{point: nice local structure}. For two algebraic spaces $U_1$ and $U_2$ as in~\eqref{point: nice local structure}, $\widetilde{\cJ}_{U_1}$ coincides with $\widetilde{\cJ}_{U_2}$ over the open intersection of the images of $U_1$ and $U_2$ in $\cXMOD$. It follows that for all $U\rightarrow\cXMOD$ as in~\eqref{point: nice local structure}, we have $\cI_{\cX}^{\circ}\times_{\cX}\cU\simeq\cJ_U$. Therefore, $\cI_{\cX}^{\circ}\subset\cI_{\cX}$ is an open and closed subgroup. By Theorem~\ref{thm: intro characterization CXDM}, there exists a good moduli space morphism $\varphi:\cX\rightarrow\cY$ with relative inertia equal to $\cI_{\cX}^{\circ}$, and $\cY$ a Deligne-Mumford stack. Notice that the stabilizers of $\cY$ satisfy property~\eqref{point: characterization stabilizers}. Using Lemma~\ref{lem: surjectivity stabilizers good and smooth stab case}, it follows easily that the induced morphism $\cXDM\rightarrow\cY$ is an isomorphism, showing that also $\cXDM$ satisfies~\eqref{point: characterization stabilizers}.
\end{proof}

A similar statement holds for the initial morphism $\Phi_{\mathrm{fi}}:\cX\rightarrow\cX_{\mathrm{fi}}$ to algebraic stacks with finite inertia. One substitutes the connected component $G_x^{\circ}$ of the stabilizer $G_x$ of $x\in\cX(k)$ with its reduced structure $(G_x)_{\circ}:=(G_x^{\circ})_{\mathrm{red}}$. Now, $(G_x)_{\circ}$ is a closed subgroup scheme of $G_x$, and it coincides with $G_x^{\circ}$ when $\mathrm{char}\ \! k=0$. When $k$ is of positive characteristic and $G_x$ is linearly reductive, $G_x^{\circ}$ is of multiplicative type, thus $(G_x)_{\circ}$ is again normal in $G_x$. See also the discussion in~\cite[Appendix B.1]{ER17}.

\begin{proposition}\label{prop: characterization commutation with base change finite inertia case}
    Let $\cX$ be an algebraic stack with affine diagonal and of finite type over an algebraically closed field $k$. Assume that it admits a good moduli space $\pi_{\cX}:\cX\rightarrow\cXMOD$, and let $\Phi_{\mathrm{fi}}:\cX\rightarrow\cX_{\mathrm{fi}}$ be the initial morphism from $\cX$ to an algebraic stack with finite inertia and affine diagonal. The following are equivalent:
    \begin{enumerate}
        \item\label{condition 1} $\Phi_{\mathrm{fi}}$ commutes with base change along morphisms with affine diagonal from algebraic stacks with finite inertia to $\cX_{\mathrm{fi}}$.
        \item\label{condition 2} $\Phi_{\mathrm{fi}}$ commutes with base change along affine étale morphisms to $\cXMOD$.
        \item\label{condition 3} for every point $x\in\cX(k)$, there exists an affine étale neighborhood $U\rightarrow\cXMOD$ of $\pi_{\cX}(x)$ such that $\cX\times_{\cXMOD}U\simeq[\Spec A/G_x]$
        and $\cX_{\mathrm{fi}}\times_{\cXMOD}U\simeq[\Spec A^{(G_x)_{\circ}}/(G_x/(G_x)_{\circ})]$.
        \item\label{condition 4} for every point $x\in\cX(k)$, with image $y:=\Phi_{\mathrm{fi}}(x)$ in $\cX_{\mathrm{fi}}$, the morphism $G_x\rightarrow G_y$ induces an isomorphism $G_x/(G_x)_{\circ}\simeq G_y$.
        \item\label{condition 5} for every point $x\in\cX(k)$, there exists an affine étale neighborhood $U\rightarrow\cXMOD$ of $\pi_{\cX}(x)$ such that $\cX\times_{\cXMOD}U\simeq[\Spec A/G_x]$
        and $(G_{x'})_{\circ}=G_{x'}\cap(G_x)_{\circ}$ for every other closed point $x'$ in $\cX\times_{\cXMOD}U$.
    \end{enumerate}
    Moreover, if the above equivalent conditions hold, then also the formation of $\PhiDM:\cX\rightarrow\cXDM$ commutes with base change along morphisms from Deligne-Mumford stacks to $\cXDM$.
\end{proposition}
\begin{proof}
    The initial morphism $\Phi_{\mathrm{fi}}$ exists by Corollary~\ref{cor: finite inertia initial map}, and has affine diagonal by Lemma~\ref{lem: inheriting affine diagonal}.
    The proof that~\eqref{condition 1} implies~\eqref{condition 2}, that~\eqref{condition 2} implies~\eqref{condition 3}, which in turn implies~\eqref{condition 4}, are as in Proposition~\ref{prop: characterization commutation with base change DM case}. The same is true for the implication from~\eqref{condition 3} to~\eqref{condition 5}. Now, we show that~\eqref{condition 4} implies~\eqref{condition 1}.
    
    Assume that~\eqref{condition 4} holds, and notice the same is true for all geometric points. Let $\cY\rightarrow\cX_{\mathrm{fi}}$ be a morphism with affine diagonal from an algebraic stack with finite inertia, and construct the commutative diagram with left cartesian square
    \[
    \begin{tikzcd}
        \cU\arrow[r,"\varphi"]\arrow[d] & \cY\arrow[r]\arrow[d] & \cYMOD\arrow[d]\\
        \cX\arrow[r,"\Phi_{\mathrm{fi}}"] & \cXDM\arrow[r] & \cXMOD
    \end{tikzcd}
    \]
    where the horizontal arrows are good moduli space morphisms by~\cite[Proposition 4.7 (1)]{goodmoduli_Alp13}. Since the left square is cartesian and $\cY$ has finite inertia, $\cU\rightarrow\cY$ again satisfies condition~\eqref{condition 4}. Let $f:\cU\rightarrow\cZ$ be a morphism to an algebraic stack with finite inertia; we want to show that $f$ factors through $\cY$. By Lemma~\ref{lem: technical throw in lemma}, we can assume $f$ to be a good moduli space morphism. By Lemma~\ref{lem: refinement property of CX}, we can further assume to be in a situation where $\varphi$ is factored as $\cU\xrightarrow{f}\cZ\xrightarrow{g}\cY$. We just need to show that $g$ induces an isomorphism between all geometric stabilizers, as it is a good moduli space morphism by Lemma~\ref{lem: cancellation properties}\eqref{lem: cancellation properties 3}. Clearly, for every geometric point $u$ in $\cU$ the kernel of $G_u\rightarrow G_{f(u)}$ contains $(G_u)_{\circ}$. By Lemma~\ref{lem: surjectivity stabilizers good and smooth stab case}, $g$ induces a surjection between stabilizers, thus these are actually isomorphism by the assumption~\eqref{condition 4}.

    Now, we prove that condition~\eqref{condition 5} implies property~\eqref{condition 4}. Given a geometric point $x$ in $\cX$, let $U\rightarrow\cXMOD$ be the neighborhood provided by~\eqref{condition 5}, and let $\cX_U:=\cX\times_{\cXMOD}U$. Then, there is a morphism $\cX_U\rightarrow[\Spec A^{(G_x)_{\circ}}/(G_x/(G_x)_{\circ})]$. From the assumption, it follows that this morphism satisfies property~\eqref{condition 4}, hence $[\Spec A^{(G_x)_{\circ}}/(G_x/(G_x)_{\circ})]\simeq(\cX_U)_{\mathrm{fi}}$ and this is true after further restriction to any open substack. Now, assume further that $U\rightarrow\cXMOD$ is also surjective. Let $R:=U\times_{\cXMOD}U$, and $\cX_R:=R\times_{\cXMOD}\cX$, which does not depend on the projection $R\rightarrow U$ chosen. Then, we get the following commutative diagram of solid arrows
    \begin{equation}\label{diag: finite inertia}
    \begin{tikzcd}
        \cX_R\arrow[d,shift right]\arrow[d,shift left]\arrow[r] & (\cX_R)_\mathrm{fi}\arrow[d,shift right]\arrow[d,shift left]\arrow[r] & R\arrow[d,shift right]\arrow[d,shift left]\\
        \cX_U\arrow[d]\arrow[r] & (\cX_U)_\mathrm{fi}\arrow[d,dashed,"\pi"]\arrow[r] & U\arrow[d]\\
        \cX\arrow[rr, bend right=20]\arrow[r,dashed,"\varphi"] & \cY\arrow[r,dashed] & \cXMOD
    \end{tikzcd}
    \end{equation}
    where the two vertical morphisms in the middle are induced by the universal property of $(\cX_R)_{\mathrm{fi}}$. Notice that $\cX_R\simeq\cX_U\times_{U}R$ under both projections from $R$ to $U$. Since $\cX_U\rightarrow(\cX_U)_{\mathrm{fi}}$ satisfies property~\eqref{condition 4}, its construction commutes with base change, hence the top right square is cartesian. All vertical arrows are stabilizer preserving and étale.
    Therefore, $(\cX_R)_{\mathrm{fi}}\rightrightarrows(\cX_U)_{\mathrm{fi}}$ forms a stabilizer-preserving étale groupoid object in the 2-category of algebraic stacks. Because $(\cX_U)_{\mathrm{fi}}$ and $(\cX_R)_{\mathrm{fi}}$ are defined as initial objects, the structural 1-morphisms of this groupoid are unique up to unique 2-isomorphisms. Consequently, all 2-categorical coherence conditions are satisfied automatically. This avoids many of the $\infty$-categorical complexities of groupoids in stacks (as studied, for instance, in~\cite[\S3]{FactorizationStacks_Har17} in the case of Deligne-Mumford stacks). In particular, by~\cite[Lemma 2.1.4]{ToroidalOrbifoldsATW20}, there exists a quotient $\cY$, and we can complete diagram~\eqref{diag: finite inertia} with the dotted arrows.
    The squares of the resulting diagram are all cartesian, and the quotient map $\pi:(\cX_U)_{\mathrm{fi}}\rightarrow\cY$ is étale and stabilizer preserving. It follows that $\cY$ has finite inertia. By the universal property, we get an induced morphism $\phi:\cX_{\mathrm{fi}}\rightarrow\cY$ factoring $\varphi$. By construction, $\cY$ has stabilizer at $\varphi(x')$ equal to $G_{x'}/(G_{x'})_{\circ}$, hence $\phi$ is an isomorphism, by Lemma~\ref{lem: surjectivity stabilizers good and smooth stab case}. Therefore, when $U\rightarrow\cXMOD$ is surjective, we have constructed $\cX_{\mathrm{fi}}$ and showed that it satisfies properties~\eqref{condition 1}-\eqref{condition 4}, hence in particular its construction commutes with open immersions. In the general case, we obtain local models of $\cX_{\mathrm{fi}}$ around each geometric point $x$, that commute with base change along open immersions, hence they glue. It follows that the global $\cX_{\mathrm{fi}}$ satisfies~\eqref{condition 5}.
    
    Finally, notice that property~\eqref{condition 5} implies $G_x'^{\circ}=G_{x'}\cap G_x^{\circ}$ in the same étale neighborhood. Indeed, both right and left-hand side are open subschemes of $G_{x'}$ with the same underlying topological space. In particular, the formation of $\PhiDM:\cX\rightarrow\cXDM$ commutes with base change by Proposition~\ref{prop: characterization commutation with base change DM case}.
\end{proof}

\begin{example}\label{exm: counterexample non-commutativity finite inertia}
    Let $k$ be a field of characteristic 2, and let $\cX=[\A^2/\Gm]$, where the action has weights $(2,-2)$. Then, $\cX_{\mathrm{fi}}\simeq\cXDM\simeq\cXMOD\simeq\A^1$. Let $x_0$ be the origin, and $x$ any other point of $\cX$. Then, $G_x=G_x^{\circ}=G_x\cap G_{x_0}^{\circ}=\mu_2$, but $(G_x)_{\circ}=\Spec k\not=\mu_2=G_x\cap (G_{x_0})_{\circ}$. Propositions~\ref{prop: characterization commutation with base change DM case} and~\ref{prop: characterization commutation with base change finite inertia case} show that the formation of $\cXDM$ commutes with base change while $\cX_{\mathrm{fi}}$ does not. Indeed, if $\cU=\cX\times_{\cXMOD}(\A^1\setminus0)$, then $\cU_{\mathrm{fi}}\simeq\cU$.
\end{example}

As an application of Proposition~\ref{prop: characterization commutation with base change finite inertia case}, we recover~\cite[Proposition B.2]{ER17} over an algebraically closed field.

\begin{corollary}[{\cite[Proposition B.2]{ER17}}]\label{cor: Edidin Rydh result}
    Let $\cX$ be a reduced algebraic stack with affine diagonal and of finite type over an algebraically closed field $k$. Assume that it admits a good moduli space $\pi_{\cX}:\cX\rightarrow\cXMOD$. If the dimension of the stabilizers of geometric points of $\cX$ is constant, then $\cXDM$ and $\cX_{\mathrm{fi}}$ satisfy the equivalent conditions of Propositions~\ref{prop: characterization commutation with base change DM case} and~\ref{prop: characterization commutation with base change finite inertia case}. Moreover, $\Phi_{\mathrm{fi}}$ is a smooth gerbe.
    In particular, if $\cX$ is smooth, then $\cX_{\mathrm{fi}}$ is smooth and $\cXMOD$ has tame quotient singularities.
\end{corollary}
\begin{proof}
    By Proposition~\ref{prop: characterization commutation with base change finite inertia case}, to prove the base change property it is enough to show that $\cX$ satisfies property~\eqref{condition 5} of the same proposition. For this, we can assume $\cX\simeq[\Spec A/G_x]$ where $G_x$ is the stabilizer group scheme of a point $x\in\cX(k)$ in $\cX$, and show that $(G_{x'})_{\circ}=G_{x'}\cap(G_{x})_{\circ}$ for every other geometric point $x'$. Since $G_{x'}\subset G_x$ and have the same dimension, $G_{x'}$ is set-theoretically the union of some connected components of $G_x$. It follows that the stronger equality $(G_{x'})_{\circ}=(G_x)_{\circ}$ holds.

    By what we have just proved and Proposition~\ref{prop: characterization commutation with base change finite inertia case}, to prove that $\Phi_{\mathrm{fi}}$ is a smooth gerbe we can again assume that $\cX\simeq[\Spec A/G_x]$ and $\cX_{\mathrm{fi}}\simeq[\Spec A^{(G_x)_{\circ}}/(G_x/(G_x)_{\circ})]$. Since $(G_{x'})_{\circ}=(G_x)_{\circ}$ for every geometric point $x'$, and $\cX$ is reduced, we have that $[\Spec A^{(G_x)_{\circ}}/(G_x/(G_x)_{\circ})]\simeq[\Spec A/(G_x/(G_x)_{\circ})]$. In particular, $\Phi_{\mathrm{fi}}$ is a smooth $(G_x)_{\circ}$-gerbe.
\end{proof}

\section{Inadequacy of Adequate Moduli Morphisms}\label{sec: counterexample adequate}

In this section, we show that the stabilization property (Proposition~\ref{prop: generalization of stabilization}) does not always hold when the moduli space morphisms are only assumed to be adequate and the targets have non-reduced stabilizers.
The construction is inspired by Nagata's example (\cite{RationalActions_Nagata},\cite[Example 6.5.1]{QuotientSpaces_Kol97}) of a Noetherian ring with a $\Z/p\Z$-action whose ring of invariants is not Noetherian.

\subsection{The Construction}

Let $k$ be an infinite field of characteristic $p$, and let $y$ be transcendental over $k$. Then, we choose general polynomials $q_i(y)\in k(y)$ such that $K_n:=k(y)[x_1,\ldots,x_n]/(x_i^p-q_i(y))$ is a degree $p$ field extension of $K_{n-1}$, where $K_0:=k(y)$. Let $K_{\infty}=\colim K_n$. For every $n$, the morphism $\Spec K_n\rightarrow\Spec K_0$ is a $\mu_p^n$-torsor, where the $i$-th factor acts on $x_i$ via multiplication. This action is compatible with the $\mu_p$-torsors $\Spec K_n\rightarrow\Spec K_{n-1}$ and the projections $\mu_p^n\rightarrow\mu_p^{n-1}$ forgetting the last factor. Therefore, there is an induced action of $\mu_p^{\infty}:=\lim\mu_p^n$ on $\Spec K_{\infty}$, making it a torsor over $K_0$.
For every $n\geq0$, set $R_n:=K_n[\epsilon]/(\epsilon^2)$, and $R_{\infty}:=\lim R_n\simeq K_{\infty}[\epsilon]/(\epsilon^2)$. By base change, there is a $\mu_p^\infty$-action on $\Spec R_{\infty}$ making it a torsor over $\Spec R_0$.

Our goal now is to construct a $\Z/p\Z$-action on $\Spec R_{\infty}$ commuting with the $\mu_p^\infty$-action, and such that $R_{\infty}$ is not finitely generated over $R_{\infty}^{\Z/p\Z}$.
Let $D:K_{\infty}\rightarrow K_{\infty}$ be the $k(y)$-derivation defined as
\begin{equation}\label{eq: def derivation}
    D(f)=\sum_i x_i^{1+p}\frac{\partial f}{\partial x_i}=\sum_i q_i(y)x_i\frac{\partial f}{\partial x_i},
\end{equation}
which is well defined as $D(x_i^p)=0=D(q_i(y))$. Then, we choose a generator $1\in\Z/p\Z$ and define a $\Z/p\Z$-action on $\Spec K_\infty$ by
\[
    n\cdot(f+\epsilon g)\mapsto f+\epsilon(g+n\cdot D(f)).
\]
Since we are working in characteristic $p$, this is indeed an action. Notice that this defines an action on $\Spec R_n$ for all $n\geq0$, which is trivial for $n=0$.

\begin{lemma}\label{lem: invariants Z/pZ action}
    Let $F_{\infty}:=\{f\in K_{\infty}\ |\ D(f)=0\}\subset K_{\infty}$.
    \begin{enumerate}
        \item $F_{\infty}$ is a field and $R_{\infty}^{\Z/p\Z}=F_{\infty}\oplus\epsilon K_{\infty}$.
        \item $F_{\infty}\subset K_{\infty}$ is an infinite field extension.
        \item In particular, $R_{\infty}$ is not of finite type over $R_{\infty}^{\Z/p\Z}$.
    \end{enumerate}
\end{lemma}
\begin{proof}
    The first part is true by construction and the defining properties of derivations.

    To prove that $F_{\infty}\subset K_{\infty}$ is an infinite extension, one can use the same argument as in~\cite[Example 6.5.1]{QuotientSpaces_Kol97} to show that the elements $x_1,x_2,\ldots$ are linearly independent, as follows. Suppose by contradiction that there exists a relation $\sum_{i=1}^m f_ix_i=0$ with $f_i\in F_{\infty}$ not all 0, and take $m\geq1$ to be the smallest integer for which such a relation exists. In particular, $m\geq2$ and $f_m\not=0$. Applying the derivation $D$ we get
    \[
        0=D\left(\sum_{i=1}^m f_ix_i\right)=\sum_{i=1}^m f_iD(x_i)=\sum_{i=1}^m (f_iq_i(y))x_i.
    \]
    Since $f_iq_i(y)\in F_{\infty}$ for all $i$, we get a new non-trivial relation $\sum_{i=1}^{m-1}(f_i(q_m(y)-q_i(y)))x_i=0$, contradicting the minimality of $m$. Notice that we are using that $q_m(y)\not=q_i(y)$ for all $i\not=m$.
    
    The last point follows from the second, by taking the quotients by the ideal $(\epsilon)$.
\end{proof}

We only need one more property to conclude the construction.

\begin{lemma}\label{lem: commutativity of actions}
    For every $n\in\mathbb{N}\cup\{\infty\}$, the $\Z/p\Z$-action on $\Spec R_{n}$ commutes with the action of $\mu_p^n$.
\end{lemma}
\begin{proof}
    We show this by working with the functor of points. Suppose given a $k(y)$-scheme $T$ and morphisms from it to $\Z/p\Z$, $\mu_p^{n}$ and $\Spec R_{n}$. The first map corresponds to a section $z\in\Gamma(T,\cO_T)$ such that $z^p=z$, the second to a collection of $t_i\in\Gamma(T,\cO_T)$ with $t_i^p=1$, while the last map corresponds to $x_i\in\Gamma(T,\cO_T)$, with $x_i^p=q_i(y)$. Then,
    \begin{align*}
        (t_i)_i\cdot(z\cdot(x_j)_j)=(t_i)_i\cdot(x_j+\epsilon x_j^{1+p})_j=(t_jx_j+t_j^{1+p}x_j^{1+p})_j=(t_jx_j+t_jx_j^{1+p})_j=z\cdot((t_i)_i\cdot(x_j)_j).
    \end{align*}
    Since both actions are $K_0$-linear, this proves the commutativity.
\end{proof}

We are ready to construct our non-algebraic relative adequate moduli space. Consider the following diagram of solid arrows
\begin{equation}\label{diag: main construction}
\begin{tikzcd}
    \Spec R_{\infty}\arrow[d,"\phi"]\arrow[r] & {[\Spec R_{\infty}/(\Z/p\Z)]}\arrow[d,dashed,"\overline{\phi}"]\arrow[r] & \Spec R_{\infty}^{\Z/p\Z}\arrow[d,dashed,"\widetilde{\phi}"]\arrow[rd]\\
    \Spec R_0\arrow[r] & \cX:=B(\Z/p\Z)_{R_0}\arrow[rr,bend right=15,"\pi_{\cX}"]\arrow[r,dashed,"\psi"] & \cY\arrow[r,dashed] & \cXMOD
\end{tikzcd}
\end{equation}
where $\phi$ is a $\mu_p^{\infty}$-torsor. Notice that we have defined $\cX$ as $\cX=B(\Z/p\Z)_{R_0}$.

Since the composite $\Spec R_{\infty}\rightarrow\Spec R_0\rightarrow B(\Z/p\Z)_{R_0}$ is $\Z/p\Z$-invariant, we can complete the left square with the dashed arrow $\overline{\phi}$. This square is cartesian, as it is commutative and with rows being $\Z/p\Z$-torsors. By Lemma~\ref{lem: commutativity of actions}, the $\mu_p^\infty$-action on $\Spec R_{\infty}$ descends to an action on $[\Spec R_{\infty}/(\Z/p\Z)]$, making $\overline{\phi}$ a $\mu_p^{\infty}$-torsor under this action.

By the universal property of adequate moduli spaces~\cite[Theorem 3.12]{LocalStructureStacks_AHR25}, the $\mu_p^\infty$-action further descends to an action on $\Spec R_{\infty}^{\Z/p\Z}$, and we let $\widetilde{\phi}:\Spec R_{\infty}^{\Z/p\Z}\rightarrow\cY$ be the associated quotient stack; this is a stack in the fpqc topology, which is not guaranteed to be algebraic (see Example~\ref{exm: example fpqc gerbe}). Since the top-right horizontal arrow is $\mu_p^\infty$-equivariant, by the definition of quotients we get an induced morphism $B(\Z/p\Z)_{R_0}\rightarrow\cY$. This completes the square on the right, which is also cartesian. Again by the universal properties of quotients, there exists an arrow $\cY\rightarrow\cXMOD$ whose composite with $\cX\rightarrow\cY$ is the moduli space map.

As we have already remarked, the stack $\cY$ is not necessarily algebraic, as we have taken the quotient by a group scheme that is not finitely presented over the base, that is, $\mu_p^\infty$. However, $\overline{\phi}$ is a representable, affine, fpqc morphism from a scheme, hence $\cY$ is an fpqc stack. Moreover, the diagonal of $\cY$ is representable, even affine; in particular, $\cY$ is quasi-algebraic (Definition~\ref{def: fpqc stacks}). See~\S\ref{subsec: fpqc stacks} for these notions and generalities on fpqc stacks.

\begin{lemma}\label{lemma: not algebraicity example 1}
    The morphism $\psi:\cX\rightarrow\cY$ is representable by algebraic stacks and $\cY$ is a non-algebraic quasi-algebraic stack. Moreover, $\psi$ is an adequate moduli space morphism that is not of finite type.
\end{lemma}
\begin{proof}
    Since $\cY$ is quasi-algebraic, $\psi:\cX\rightarrow\cY$ is representable by algebraic stacks by Lemma~\ref{lem: quasi-algebraicity and representability by algebraic stacks}. In particular, it makes sense to say that $\psi$ is an adequate moduli space morphism, which follows by definition and diagram~\eqref{diag: main construction}, and to ask whether it is of finite type or not; see~\S\ref{subsec: finiteness properties}. Suppose $\psi$ were of finite type, then the same would hold for $\Spec R_0\rightarrow\cY$. By base change along $\widetilde{\phi}$, this would imply that $R_{\infty}$ is finitely generated over $R_{\infty}^{\Z/p\Z}$, which contradicts Lemma~\ref{lem: invariants Z/pZ action}.    
    Finally, suppose by contradiction that $\cY$ is algebraic, and let $Y\rightarrow\cY$ and $X\rightarrow\cX\times_{\cY}Y$ be smooth (or fppf) presentations. Then, $X\rightarrow\cXMOD$ is again of finite type, hence the same holds for $X\rightarrow Y$, by~\cite[Theorem 6.3.3]{AdequateModAlp14}. By definition, it follows that $\cX\rightarrow\cY$ is of finite type, contradicting our earlier conclusion.
\end{proof}

We are ready to prove Theorem~\ref{thm: intro main counterexample}.

\begin{proof}[proof of Theorem~\ref{thm: intro main counterexample}]
    We have already constructed an adequate moduli space morphism $\psi:\cX\rightarrow\cY$ to a non-algebraic fpqc stack, so we just need to construct the tower of adequate moduli space morphisms and show that its limit is $\cY$. Indeed, this chain cannot stabilize as otherwise $\cY$ would automatically be algebraic.

    The idea is to apply the same construction with $R_{\infty}$ replaced by $R_n$, and $\mu_p^{\infty}$ replaced by $\mu_p^n$. Indeed, by Lemma~\ref{lem: commutativity of actions} the actions of $\Z/p\Z$ and $\mu_p^n$ on $\Spec R_n$ commute, hence we can construct a diagram analogous to~\eqref{diag: main construction}, with $R_{\infty}$ replaced by $R_n$. For every $\infty\geq m\geq n\geq0$, the $\mu_p^{m-n}$-torsor $\Spec R_m\rightarrow\Spec R_n$ is equivariant with respect to both the $\Z/p\Z$-action and the projection $q_{m,n}:\mu_p^m\rightarrow\mu_p^n$ to the first $n$ factors. Therefore, we obtain a morphism $\Spec R_m^{\Z/p\Z}\rightarrow\Spec R_n^{\Z/p\Z}$ that is equivariant with respect to $q_{m,n}$. Let $\cY_n:=[\Spec R_n^{\Z/p\Z} / \mu_p^n]$. We obtain a tower of morphisms
    \begin{equation}\label{diag: tower relative adequate moduli space}
    \begin{tikzcd}
        \cX\arrow[r] & \cY\arrow[r] & \ldots\arrow[r] & \cY_n\arrow[r] & \cY_{n-1}\arrow[r] & \ldots\arrow[r] & \cY_1\arrow[r] & \cXMOD
    \end{tikzcd}
    \end{equation}
    that factor the adequate moduli space morphism $\cX\rightarrow\cXMOD$. Notice that all $\cY_i$ are algebraic stacks with affine diagonal. Since the base change of $\psi_n:\cX\rightarrow\cY_n$ along the fppf cover $\Spec R_n^{\Z/p\Z}\rightarrow\cY_n$ is the adequate moduli space morphism $[\Spec R_n/(\Z/p\Z)]\rightarrow\Spec R_n^{\Z/p\Z}$, by~\cite[Proposition 5.2.9]{AdequateModAlp14} also $\psi_n$ is a relative adequate moduli space. By the same argument, each $\cY_m\rightarrow\cY_n$ is a relative adequate moduli space for all $m\geq n$.

    Now, we show that $\lim\cY_n\simeq\cY$. First of all, by~\cite[Proposition 2.15]{LocalStructureStacks_AHR25}, we have a natural isomorphism $\Spec R_{\infty}^{\Z/p\Z}\simeq\lim\Spec R_n^{\Z/p\Z}$. Since the morphisms $\Spec R_n^{\Z/p\Z}\rightarrow\cY_n$ form a compatible system of $\mu_p^n$-torsors, the induced morphism $\Spec R_{\infty}^{\Z/p\Z}\rightarrow\lim\cY_n$ is a $\mu_p^{\infty}$-torsor. As $\Spec R_{\infty}^{\Z/p\Z}\rightarrow\cY$ is also a $\mu_p^{\infty}$-torsor, it follows that $\cY\rightarrow\lim\cY_n$ is an isomorphism.
\end{proof}

\begin{remark}\label{rmk: algebraic closure}
    Notice that we can base change the whole construction along the algebraic closure morphism $\varphi:\Spec\overline{k(y)}\rightarrow\Spec k(y)$, and the chain would still not stabilize. Indeed, $\varphi$ is fpqc, hence $\cY_n\rightarrow\cY_{n-1}$ is an isomorphism if and only if its base change along $\varphi$ is. Moreover, the base change of the limit $\cY$ also remains non-algebraic. Indeed, if it were algebraic, then $\cX_{\overline{k(y)}}\rightarrow\cY_{\overline{k(y)}}$ would be of finite type by~\cite[Theorem 6.3.3]{AdequateModAlp14}, hence the same would hold for $\cX\rightarrow\cY$ by fpqc descent, contradicting Lemma~\ref{lemma: not algebraicity example 1}.
    Moreover, $\Spec\!(R_n\otimes_{k(y)}\overline{k(y)})\simeq\mu_p^n\times\Spec\overline{k(y)}$, by the change of coordinates $x_i\mapsto x_i\cdot\sqrt[p]{q_i(y)}^{-1}$ for some $p$-th root $\sqrt[p]{q_i(y)}$ of $q_i(y)$. The action of $\Z/p\Z$ on $\mu_p^{\infty}\times\Spec\overline{k(y)}[\epsilon]/(\epsilon^2)$ is again defined by the derivation in~\eqref{eq: def derivation}.
\end{remark}

\bibliographystyle{amsalpha}
\bibliography{library}

$\,$\
\noindent

$\,$\
\noindent
\textsc{Department of Pure Mathematics, Brown University, 151 Thayer Street, Providence, RI 02912, USA}

\textit{e-mail address:} \href{mailto:alberto_landi@brown.edu}{alberto\_landi@brown.edu}

\end{document}